\documentclass[notitlepage,11pt,reqno]{amsart}
\usepackage{amssymb,nicefrac,cite,bm,upgreek,verbatim}
\usepackage[mathscr]{eucal}
\usepackage{graphicx}
\usepackage{wrapfig}
\usepackage{color}
\usepackage{hyperref}

\allowdisplaybreaks

\usepackage[margin=1in]{geometry}

\definecolor{dmagenta}{rgb}{.4,.1,.5}
\definecolor{dblue}{rgb}{.0,.0,.6}
\definecolor{dred}{rgb}{.5,.0,.0}
\definecolor{dgreen}{rgb}{.0,.4,.0}
\definecolor{005}{rgb}{.0,.0,.5}

\definecolor{violet}{rgb}{.3,.0,.9}
\definecolor{dmag}{cmyk}{.0,.0,1,.5}
\definecolor{dcyan}{cmyk}{1,.0,.0,.5}
\definecolor{cm}{cmyk}{1,.0,.0,.0}
\definecolor{dyellow}{cmyk}{.0,.0,.5,.0}

\newtheorem{lemma}{Lemma}[section]
\newtheorem{theorem}{Theorem}[section]

\newtheorem{corollary}{Corollary}[section]

\theoremstyle{definition}
\newtheorem{definition}{Definition}[section]
\newtheorem{assumption}{Assumption}[section]
\newtheorem{hypothesis}{Hypothesis}[section]
\newtheorem*{algo}{Leaf Elimination Algorithm}
\newtheorem{example}{Example}[section]

\theoremstyle{remark}
\newtheorem{remark}{Remark}[section]
\numberwithin{equation}{section}

\newcommand{\pex}{(e\cdot x)^{+}}
\newcommand{\nex}{(e\cdot x)^{-}}
\newcommand{\him}{{\Hat\imath}}
\newcommand{\hjm}{{\Hat\jmath}}

\newcommand{\cA}{{\mathcal{A}}}
\newcommand{\cB}{{\mathcal{B}}} 
\newcommand{\cE}{{\mathcal{E}}} 
\newcommand{\eom}{{\mathscr{G}}}
\newcommand{\cG}{{\mathcal{G}}}  
\newcommand{\sH}{{\mathscr{H}}}  
\newcommand{\cH}{{\mathcal{H}}}  
\newcommand{\cI}{{\mathcal{I}}}  
\newcommand{\cJ}{{\mathcal{J}}}  
\newcommand{\cK}{{\mathcal{K}}} 
  
\newcommand{\Lg}{\mathcal{L}}    
\newcommand{\cL}{{\mathscr{L}}} 

\newcommand{\cM}{{\mathcal{M}}}  
\newcommand{\cP}{{\mathcal{P}}}  

\newcommand{\sR}{{\mathscr{R}}}
\newcommand{\cS}{{\mathcal{S}}}

\newcommand{\sU}{{\mathscr{U}}}
\newcommand{\Lyap}{{\mathcal{V}}}
\newcommand{\sX}{{\mathscr{X}}}
\newcommand{\cX}{{\mathcal{X}}}
\newcommand{\fZ}{{\mathfrak{Z}}}

\newcommand{\RR}{\mathbb{R}}
\newcommand{\NN}{\mathbb{N}}
\newcommand{\ZZ}{\mathbb{Z}}
\newcommand{\RI}{\mathbb{R}^{I}}
\newcommand{\Rd}{\mathbb{R}^{d}}
\DeclareMathOperator{\Exp}{\mathbb{E}}
\DeclareMathOperator{\Prob}{\mathbb{P}}
\newcommand{\D}{\mathrm{d}}
\newcommand{\E}{\mathrm{e}}

\newcommand{\Act}{{\mathbb{U}}}
\newcommand{\Uadm}{\mathfrak{U}}
\newcommand{\Usm}{\mathfrak{U}_{\mathrm{SM}}}
\newcommand{\Ussm}{\mathfrak{U}_{\mathrm{SSM}}}

\newcommand{\sF}{\mathfrak{F}}
\newcommand{\Ind}{\mathbb{I}}
\newcommand{\Cc}{\mathcal{C}} 
\newcommand{\Ccl}{\mathcal{C}_{\mathrm{loc}}}

\newcommand{\Sobl}{\mathscr{W}_{\mathrm{loc}}}

\newcommand{\abs}[1]{\lvert#1\rvert}
\newcommand{\norm}[1]{\lVert#1\rVert}
\newcommand{\babs}[1]{\bigl\lvert#1\bigr\rvert}

\newcommand{\babss}[1]{\biggl\lvert#1\biggr\rvert}

\newcommand{\transp}{^{\mathsf{T}}}
\newcommand{\df}{:=}

\DeclareMathOperator*{\Argmin}{Arg\,min}
\DeclareMathOperator*{\diag}{diag}
\DeclareMathOperator*{\trace}{trace}

\newcommand{\order}{{\mathscr{O}}}
\newcommand{\sorder}{{\mathfrak{o}}}
\newcommand{\grad}{\nabla}
\newcommand{\tc}{{\Breve\uptau}}

\newcommand{\beql}[1]{\begin{equation}\label{#1}}
\newcommand{\eeq}{\end{equation}}

\begin{document}


\title[Ergodic Diffusion Control of Multiclass Multi-Pool Networks]
{Ergodic Diffusion Control of Multiclass Multi-Pool
Networks in the Halfin--Whitt Regime}

\author{Ari Arapostathis}
\address{Department of Electrical and Computer Engineering,
The University of Texas at Austin,
1 University Station, Austin, TX 78712}
\email{ari@ece.utexas.edu}
\author{Guodong Pang}
\address{The Harold and Inge Marcus Department of Industrial and
Manufacturing Engineering,
College of Engineering,
Pennsylvania State University,
University Park, PA 16802}
\email{gup3@psu.edu}

\date{\today}

\begin{abstract}
We consider Markovian multiclass multi-pool networks with heterogeneous server
pools, each consisting of many statistically identical parallel servers,
where the bipartite graph of customer classes and server pools forms a tree. 
Customers form their own queue and are served in the first-come first-served
discipline, and can abandon while waiting in queue. Service rates are both
class and pool dependent. The objective is to study the
limiting diffusion control problems under the long run
average (ergodic) cost criteria in the Halfin--Whitt regime. 
Two formulations of ergodic diffusion control problems are
considered: (i) both queueing and idleness costs are minimized, and (ii) only
the queueing cost is minimized while a constraint is imposed upon the
idleness of all server pools.  
We develop a \emph{recursive leaf elimination algorithm} that enables us to
obtain an explicit representation of the drift for the controlled diffusions.
Consequently, we show that for the limiting controlled diffusions, there always exists
a stationary Markov control under which the diffusion process
is geometrically ergodic. 
The framework developed in \cite{ABP14} is extended to address a
broad class of ergodic diffusion control problems with constraints.
We show that that the unconstrained and constrained problems are well
posed, and we characterize the optimal stationary
Markov controls via  HJB equations.

\end{abstract}

\subjclass[2000]{60K25, 68M20, 90B22, 90B36}

\keywords{multiclass multi-pool Markovian queues, reneging/abandonment, 
Halfin--Whitt (QED) regime, controlled diffusion,  
long time average control, ergodic control,  ergodic control with constraints, 
stable Markov optimal control, spatial truncation}

\maketitle

\section{Introduction}

Consider a multiclass parallel server networks with  $I$ classes of customers (jobs) and $J$
parallel server pools, each of which has many statistically identical servers.
Customers of each class can be served in a subset of the server pools, and each
server pool can serve a subset of the customer classes, which forms a
bipartite graph.
We assume that this bipartite graph is a tree. 
  Customers of each class arrive according to a Poisson process and
form their own queue. They are served in the first-come-first-served (FCFS)
discipline. 
Customers waiting in queue  may renege if their patience times are reached
before entering service. The patience times are exponentially distributed
with class-dependent rates, while  
the service times are also exponentially distributed with rates depending on
both the customer class and the server pool. 
The scheduling and routing control decides which class of customers to serve
(if any waiting in queue) when a server becomes free, and which server pool to
route a customer when multiple server pools have free servers to serve the
customer. 
 We  focus on preemptive scheduling policies that satisfy the
usual work conserving condition (no server can idle if a customer it can serve
is in queue), as well as the joint work conserving condition
\cite{Atar-05a, Atar-05b,  Atar-09} under which, customers can be rearranged
in such a manner that no server will idle when a customer of some class is
waiting in queue.   
In this paper, we study the diffusion control problems of such multiclass multi-pool networks under the long run average (ergodic) cost criteria in the
Halfin--Whitt regime. 

We consider two formulations of the ergodic diffusion control problems. 
In the first formulation, both queueing and idleness are penalized in the running cost, and we refer to this as the ``\emph{unconstrained}" problem. In the second formulation, only the queueing cost is minimized, while a constraint is imposed upon the idleness of all server pools. We refer to this as the ``\emph{constrained}" problem. The constraint can be regarded as a ``fairness" condition on server pools. 
We aim to study the recurrence properties of the controlled diffusions,  the
well-posedness of  these two ergodic diffusion control problems,
and characterize the optimal stationary Markov controls via
 Hamilton--Jacobi--Bellman (HJB) equations. 

The diffusion limit of the queueing processes for the multiclass multi-pool networks was established by Atar \cite{Atar-05a, Atar-05b}. Certain properties of the controlled diffusions were proved in \cite{Atar-05a, Atar-09}, with
the objective of studying the diffusion control problem under the discounted cost criterion. However those properties do not suffice for the study the ergodic control problem.
Our first task is to obtain a good understanding of the recurrence
properties of the limiting controlled diffusions. The main obstacle lies in the \emph{implicitness} of the drift, which is represented via the solution of a linear program (Section~\ref{sec-diffusionlimit}). 
Our first key contribution is to provide an \emph{explicit} representation
of the drift of the limiting controlled diffusions via a recursive leaf elimination
algorithm (Sections~\ref{sec-leaf-alg} and \ref{sec-example}).
As a consequence, we show that the controlled diffusions have a piecewise linear drift
(Lemma~\ref{lem-drift}), which, unfortunately, does not belong to the class
of piecewise linear diffusions studied in \cite{dieker-gao} and \cite{ABP14},
despite the somewhat similar representations. The dominating matrix in the drift is a Hurwitz lower-diagonal matrix, instead of the negative of an $M$-matrix. 
Applying the leaf elimination algorithm, we show that for any Markovian multiclass multi-pool (acyclic)
network in the Halfin--Whitt regime, assuming that the abandonment rates are not identically zero, 
there exists a stationary Markov control under which the limiting diffusion
is geometrically ergodic, and as a result, its invariant probability distribution
has all moments finite  (Theorem~\ref{T-stab}).

A new framework to study ergodic diffusion control problems was introduced in \cite{ABP14}, in order to study the multiclass single-pool network (the ``V" model) in the Halfin-Whitt regime. It imposes a structural assumption 
(Hypothesis~3.1),
which extends the applicability of the theory beyond the two dominant
models in the study of ergodic control for diffusions \cite{book}:
(i) the running cost is
\emph{near-monotone} and (ii) the controlled diffusion is \emph{uniformly stable}. The relevant results are reviewed in Section~\ref{sec-structural}. 
Like the ``V" model, the ergodic control problems of diffusions associated
with multiclass
multi-pool networks do not fall into any of those two categories.  
We show that the ``unconstrained" ergodic diffusion control problem is well-posed and
can be solved using the framework in \cite{ABP14}.
Verification of the structural assumption in Hypothesis~3.1, relies heavily upon the explicit representation of the drift in the limiting controlled diffusions
(Theorem~\ref{T-PA}).
We then establish the existence  of an optimal stationary Markov control,
characterize all such controls
via an HJB equation in Section~\ref{sec-HJBunconstrained}. 

Ergodic control with constraints for diffusions was studied in
\cite{BorkGh-90a,Borkar93a}; see Sections~4.2 and 4.5 in \cite{book}. 
However, the existing methods and theory also fall into the same two categories
mentioned above.
Therefore, to study the well-posedness and solutions of the ``constrained" problem, we
extend the framework in \cite{ABP14} to ergodic diffusion control
problems with constraints under the same structural assumption
in Section~\ref{sec-constraint}.
The
well-posedness of the constrained problem follows by Lemma~\ref{L3.3}
of that section. 
We also characterize the optimal
stationary controls via an HJB equation, which has a unique solution
in a certain class (Theorems~\ref{T3.1} and \ref{T-unique}).
We also extend the ``spatial truncation" technique developed in \cite{ABP14}
to  problems under constraints (Theorems~\ref{T3.3} and \ref{T3.4}). 
These results are applied to the ergodic diffusion control problem with constraints
for the multiclass multi-pool networks in Section~\ref{sec-HJBconstrained}. 
The special case of fair allocation of idleness in
the constrained problem is discussed in Section~\ref{sec-fair}. 

It is worth noting that if we only penalize the queue but not the idleness,
the unconstrained ergodic control problem
may not be well-posed.
We discuss the verification of the structural assumption (Hypothesis~3.1),
in this formulation of the ergodic diffusion control problem in
Section~\ref{sec-special}.
We show that under certain restrictions on the systems parameters or network structure,
Hypothesis~3.1 can be verified and this formulation is therefore well-posed
(see Corollaries~\ref{C4.1} and \ref{C4.2}, and Remark~\ref{R4.6}).

\subsection{Literature review}

Scheduling and routing control of multiclass multi-pool networks in the
Halfin--Whitt has been studied extensively in the recent literature.
Atar \cite{Atar-05a, Atar-05b} was the first to study scheduling and routing control
problem under infinite-horizon discounted cost. He has  solved
the scheduling control problem under a set of conditions on the network structure
and parameters, and the running cost function (Assumptions~2 and 3 in
\cite{Atar-05b}).
Simplified models with either class only, or pool only dependent service
rates under the infinite-horizon discounted cost are further studied in
Atar et al. \cite{Atar-09}. 
Gurvich and Whitt \cite{GW09, GW09b, GW10} studied queue-and-idleness-ratio
controls and their associated properties and staffing implications for
multiclass multi-pool networks, by proving a state-space-collapse (SSC) property
under certain conditions on the network structure and system parameters
(Theorems~3.1 and 5.1 in \cite{GW09}). 
Dai and Tezcan \cite{dai-tezcan-08, dai-tezcan-11} studied scheduling
controls of multiclass multi-pool networks in the finite-time horizon,
also by proving an SSC property under certain assumptions.
Despite all these results that have helped us better understand the performance
of a large class of multiclass multi-pool networks, there is a lack of good
understanding of the behavior of the limiting controlled diffusions due to
the \emph{implicit} form of its drift. 
Our result on an explicit representation of the drift breaks this fundamental barrier.

There is limited literature on ergodic control of multiclass multi-pool networks
in the Halfin--Whitt regime.
Ergodic control of the multiclass ``V" model is recently studied in \cite{ABP14}.
Armony \cite{Armony-05} studied the inverted ``V" model and showed that the
fastest-server-first policy is asymptotically optimal for minimizing the
steady-state expected queue length and waiting time.
Armony and Ward \cite{AW-10} showed that for the inverted ``V" model,
a threshold policy is asymptotically optimal for minimizing the steady-state
the expected queue length and waiting time subject to a ``fairness" constraint
on the workload division. 
Ward and Armony \cite{WA-13} studied blind fair routing policies for multiclass
multi-pool networks, which is based on the number of customers waiting and the
number of severs idling but not on the system parameters, and used simulations to
validate the performance of the blind fair routing policies comparing them with
non-blind policies derived from the limiting diffusion control problem.
Biswas \cite{Biswas-15} recently studied a multiclass multi-pool network with
``help" where each server pool has a dedicated stream of a customer class,
and can help with other customer classes only when it has idle servers.
In such a network, the control policies may not be work-conserving, and from
the technical perspective, the associated controlled diffusion has a
uniform stability property, which is not satisfied for general multiclass
multi-pool networks.

\subsection{Organization}
The rest of this section contains a summary of the notation used in the paper. 
In Section~\ref{sec-MM-model}, we introduce the multiclass multi-pool
parallel server network model, the asymptotic Halfin--Whitt regime,
the state descriptors and the admissible scheduling and routing controls.
In Section~\ref{sec-MM-control-hw}, we introduce the diffusion-scaled processes in the Halfin-Whitt regime and the associated control parameterization, and in Section~\ref{sec-diffusionlimit} we state the limiting controlled diffusions.
In Section~\ref{sec-MM-control-diff}, we describe the two formulations of the ergodic diffusion control problems. 
In Section~\ref{S3}, we first review the general model
of controlled diffusions studied in \cite{ABP14}, and then state the
general hypotheses and the associated stability results
(Section~\ref{sec-structural}).
We then study the associated ergodic control problems with constraints
in Section~\ref{sec-constraint}. 
We focus on the recurrence properties of the controlled diffusions for
multiclass multi-pool networks in Section~\ref{sec-stabilizability}.
The leaf elimination algorithm and the resulting drift representation are
introduced in Section~\ref{sec-leaf-alg}, and some examples applying the
algorithm are given in Section~\ref{sec-example}.
We verify the structural assumption of Section~\ref{sec-structural} and study the positive recurrence properties of the limiting controlled diffusions
in Section~\ref{sec-verify}. 
We discuss
some special cases in Section~\ref{sec-special}.
The optimal stationary Markov controls for the limiting diffusions are
characterized in Section~\ref{sec-HJB}.
Some concluding remarks are given in
Section~\ref{sec-conclusion}. 

\subsection{Notation}
The following notation is used in this paper.
The symbol $\RR$, denotes the field of real numbers,
and $\RR_{+}$ and $\NN$ denote the sets of nonnegative
real numbers and natural numbers, respectively.
Given two real numbers $a$ and $b$, the minimum (maximum) is denoted by $a\wedge b$ 
($a\vee b$), respectively.
Define $a^{+}\df a\vee 0$ and $a^{-}\df-(a\wedge 0)$. 
The integer part of a real number $a$ is denoted by $\lfloor a\rfloor$.
We use the notation 
$e_{i}$, $i=1,\dotsc,d$, to denote the vector with $i^\text{th}$ entry equal to $1$
and all other entries equal to $0$.
We also let $e\df (1,\dotsc,1)\transp$.

For a set $A\subset\Rd$, we use
$\Bar A$, $A^{c}$, $\partial A$, and $\Ind_{A}$ to denote the closure,
the complement, the boundary, and the indicator function of $A$, respectively.
A ball of radius $r>0$ in $\Rd$ around a point $x$ is denoted by $B_{r}(x)$,
or simply as $B_{r}$ if $x=0$.
The Euclidean norm on $\Rd$ is denoted by $|\cdot|$,
$x\cdot y$,
denotes the inner product of $x,y\in\RR^{d}$,
and $\norm{x}\df \sum_{i=1}^{d}\abs{x_{i}}$.

For a nonnegative function $g\in\Cc(\RR^{d})$ 
we let $\order(g)$ denote the space of functions
$f\in\Cc(\RR^{d})$ satisfying
$\sup_{x\in\RR^{d}}\;\frac{\abs{f(x)}}{1+g(x)}<\infty$.
This is a Banach space under the norm
\begin{equation*}
\norm{f}_{g} \;\df\;
\sup_{x\in\RR^{d}}\;\frac{\abs{f(x)}}{1+g(x)}\,.
\end{equation*}
We also let $\sorder(g)$ denote the subspace of $\order(g)$ consisting
of those functions $f$ satisfying
\begin{equation*}
\limsup_{\abs{x}\to\infty}\;\frac{\abs{f(x)}}{1+g(x)}\;=\;0\,.
\end{equation*}
Abusing the notation, $\order(x)$ and $\sorder(x)$ occasionally denote
generic members of these sets.
For two nonnegative functions $f$ and $g$, we use the notation $f\sim g$
to indicate that $f\in\order(g)$ and $g\in\order(f)$.

We denote by $L^{p}_{\mathrm{loc}}(\RR^{d})$, $p\ge 1$, the set of
real-valued functions
that are locally $p$-integrable and by
$\Sobl^{k,p}(\RR^{d})$ the set of functions in $L^{p}_{\mathrm{loc}}(\RR^{d})$
whose $i^\text{th}$ weak derivatives, $i=1,\dotsc,k$, are in
$L^{p}_{\mathrm{loc}}(\RR^{d})$.
The set of all bounded continuous functions is denoted by $\Cc_{b}(\RR^{d})$.
By $\Ccl^{k,\alpha}(\RR^{d})$ we denote the set of functions that are
$k$-times continuously differentiable and whose $k^\text{th}$ derivatives are locally
H\"{o}lder continuous with exponent $\alpha$.
We define $\Cc^{k}_{b}(\RR^{d})$, $k\ge 0$, as the set of functions
whose $i^\text{th}$ derivatives, $i=1,\dotsc,k$, are continuous and bounded
in $\RR^{d}$ and denote by
$\Cc^{k}_{c}(\RR^{d})$ the subset of $\Cc^{k}_{b}(\RR^{d})$ with compact support.
For any path $X(\cdot)$
we use the notation $\Delta X(t)$ to denote the jump at time $t$.

Given any Polish space $\cX$, we denote by $\cP(\cX)$ the set of
probability measures on $\cX$ and we endow $\cP(\cX)$ with the
Prokhorov metric.
Also $\cB(\cX)$ denotes its Borel $\sigma$-algebra.
By $\delta_{x}$ we denote the Dirac mass at $x$.
For $\nu\in\cP(\cX)$ and a Borel measurable map $f\colon\cX\to\RR$,
we often use the abbreviated notation
\begin{equation*}\nu(f)\;\df\; \int_{\cX} f\,\D{\nu}\,.\end{equation*}
The quadratic variation of a square integrable martingale is
denoted by $\langle\,\cdot\,,\cdot\,\rangle$ and the optional quadratic
variation by
$[\,\cdot\,,\cdot\,]$. For presentation purposes we use the time variable 
as the subscript for the diffusion processes. 
Also $\kappa_{1},\kappa_{2},\dotsc$ and $C_{1},C_{2},\dotsc$
are used as generic constants whose values might vary from place to place.

\medskip
\section{Controlled Multiclass Multi-Pool Networks in the Halfin--Whitt Regime}

\subsection{The multiclass multi-pool network model} \label{sec-MM-model}
All stochastic variables introduced below are defined on a complete
probability space 
$(\Omega,\mathfrak{F},\Prob)$.  The expectation w.r.t. $\Prob$ is denoted by
$\Exp$. We consider a sequence of network systems with the associated variables,
parameters and processes indexed by $n$. 

Consider a multiclass multi-pool Markovian network with $I$ classes of customers
and $J$ server pools.  
The classes are labeled as $1,\dots, I$ and the server pools as $1,\dots, J$.
Set $\cI = \{1,\dots,I\}$ and $\cJ = \{1, \dots, J\}$. 
Customers of each class form their own queue and are served in the
first-come-first-served (FCFS) service discipline.
The buffers of all classes are assumed to have infinite capacity.
Customers can abandon/renege while waiting in queue. 
Each class of customers can be served by a subset of server pools,
and each server pool can serve a subset of customer classes.
For each $i \in \cI$, let $\cJ(i) \subset \cJ$
be the subset of server pools that can serve class $i$ customers, and for each
$j \in \cJ$, let $\cI(j) \subset \cI$ be the subset
of customer classes that can be served by server pool $j$.
For each $i \in \cI$ and $j \in \cJ$, if customer class $i$
can be served by server pool $j$, we denote $i \sim j$ as an edge in the bipartite
graph formed by the nodes in $\cI$ and $\cJ$;
otherwise, we denote $i \nsim j$.
Let $\cE$ be the collection of all these edges.
Let $\cG = (\cI\cup \cJ, \cE)$ be the
bipartite graph formed by the nodes (vertices) $\cI\cup \cJ$
and the edges $\cE$. We assume that the graph $\cG$ is connected. 

For each $j \in \cJ$, let $N_j^{n}$ be the number of servers
(statistically identical) in server pool $j$.
Customers of class $i \in \cI$ arrive according to a Poisson
process with rate
$\lambda^{n}_i>0$, $i \in \cI$,
and have class-dependent exponential abandonment rates $\gamma_i^{n} \ge 0$.
These customers are served at an exponential rate
$\mu_{ij}^{n}>0$ at server pool $j$, if $i \sim j$, and otherwise,
we set $\mu_{ij}^{n}=0$. 
We assume that the customer arrival, service, and abandonment processes of
all classes are mutually independent. 
The edge set $\cE$ can thus be written as
\begin{equation*}
\cE\;=\;\bigl\{(i, j) \in \cI\times \cJ\;\colon\, \mu_{ij}^{n} >0\bigr\}\,.
\end{equation*}
A pair $(i, j) \in \cE$ is called an \emph{activity}. 

\smallskip
\subsubsection{The Halfin--Whitt regime} 
We study these multiclass multi-pool networks in the Halfin--Whitt regime
(or the Quality-and-Efficiency-Driven (QED) regime), where the arrival
rates of each class and the numbers of servers of each server pool grow
large as $n \to \infty$ in such a manner that the system becomes
critically loaded. 
In particular, the set of parameters is assumed to satisfy the following:
as $n\to\infty$, the following limits exist
\begin{equation}\label{HWpara1}
\begin{split}
\frac{\lambda^{n}_{i}}{n} \;\to\;\lambda_{i}\;>\;0\,,\qquad
\frac{N^{n}_{j}}{n} \;\to\;\nu_{j}\;>\;0\,,\qquad
\mu_{ij}^{n} \;\to\;\mu_{ij}\;\ge\;0\,,\qquad
\gamma_{i}^{n} \;\to\;\gamma_{i}\;\ge \;0\,,
\end{split}
\end{equation}
\begin{equation}\label{HWpara2}
\begin{split}
\frac{\lambda^{n}_{i} - n \lambda_{i}}{\sqrt n} \;\to\;\Hat{\lambda}_{i}\,,\qquad
{\sqrt n}\,(\mu^{n}_{ij} - \mu_{ij}) \;\to\;\Hat{\mu}_{ij}\,,
\qquad {\sqrt n}\,(n^{-1} N^{n}_{j} -  \nu_{j}) \;\to\; 0 \,,   \\[5pt]
\end{split}
\end{equation}
where $\mu_{ij}>0$ for $i\sim j$ and $\mu_{ij}=0$ for $i\nsim j$. 
Note that we allow the abandonment rates to be zero for some, but not for all
$i \in \cI$.

In addition, we assume that there exists a unique optimal solution
$(\xi^*, \rho^*)$  satisfying 
\begin{equation} \label{criticalfluid}
\sum_{i \in \cI} \xi^*_{ij} \;=\; \rho^*\;=\; 1, \quad \forall j \in \cJ \,,
\end{equation}
and $ \xi^*_{ij}>0$ for all $i \sim j $ (all activities) in
$\cE$, to the following linear program (LP):
\begin{align*}
\text{Minimize} \quad & \rho \\[5pt]
\text{subject to} \quad & \sum_{j \in \cJ} \mu_{ij} \nu_j \xi_{ij}
\;=\; \lambda_i, \quad i \in \cI, \nonumber \\ 
& \sum_{i \in \cI} \xi_{ij} \;\le\; \rho, \quad j \in \cJ, \nonumber \\
& \xi_{ij} \;\ge\; 0, \quad  i \in \cI,~  j \in \cJ .
\end{align*}
This assumption is referred to as the \emph{complete resource pooling} condition
\cite{williams-2000, Atar-05b}. 
It implies that the graph $\cG$ is a tree \cite{williams-2000, Atar-05b}.
Following the terminology in \cite{williams-2000, Atar-05b}, this assumption
also implies that all activities in $\cE$ are \emph{basic} since
$\xi^*_{ij}>0$ for each activity $(i,j)$ or edge $i \sim j$ in $\cE$.
Note that in our setting all activities are basic.

We define the vector $x^*=(x_{i}^{*})_{i \in \cI}$ and matrix
$z^* = (z_{ij}^*)_{ i \in \cI, \,  j \in \cJ}$ by
\begin{equation} \label{staticfluid}
x_{i}^{*}\;=\;\sum_{j \in \cJ} \xi^*_{ij} \nu_j\,,
\qquad z_{ij}^* = \xi_{ij}^* \nu_j \,. 
\end{equation}
The vector  $x^*=(x_{i}^{*})$ can be interpreted as the steady-state total number
of customers in each class, and the matrix $z^*$ as the steady-state number
of customers in each class receiving service, in the fluid scale.
Note that the steady-state queue lengths are all zero in the fluid scale.
The solution $\xi^*$ to the LP is the steady-state proportion of customers
in each class at each server pool. It is evident that \eqref{criticalfluid}
and \eqref{staticfluid} imply that
\begin{equation*}
e\cdot  x^*= e \cdot  \nu, 
\end{equation*}
where $\nu\df(\nu_j)_{j \in \cJ}$. 

\smallskip
\subsubsection{The state descriptors} 
For each $i \in \cI$ and $j \in \cJ$, let
$X^{n}_i = \{X^{n}_i(t): t\ge 0\}$ be the total number of class $i$ customers
in the system, $Q^{n}_i = \{Q^{n}_i(t): t\ge 0\}$ be the number of class $i$
customers in the queue, $Z_{ij}^{n} = \{Z_{ij}^{n}(t): t\ge 0\}$ be the
number of class $i$ customers being served in server pool $j$, and
$Y^{n}_j = \{Y^{n}_i(t): t\ge 0\}$ be the number of idle servers in server pool $j$.
Set $X^{n} = (X_i^{n})_{i \in \cI}$, $Y^{n} = (Y_i^{n})_{i \in \cI}$, 
$Q^{n} = (Q_i^{n})_{i \in \cI}$,
and $Z^{n} = (Z_{ij}^{n})_{i \in \cI,\, j \in \cJ}$. 
The following fundamental equations hold:
for each $i \in \cI$ and $j \in \cJ$ and $t\ge 0$, we have
\begin{align} \label{baleq}
& X^{n}_i(t)\;=\;Q_i^{n}(t) + \sum_{j \in \cJ(i)} Z_{ij}^{n}(t)\,, \nonumber\\[5pt]
& N_j^{n}\;=\;Y_j^{n}(t) + \sum_{i \in\cI(j)} Z_{ij}^{n}(t)\,,  \\[5pt]
& X^{n}_i(t) \ge 0\,, \quad Q_i^{n}(t) \ge 0\,, \quad Y_j^{n}(t) \ge 0\,,
\quad Z_{ij}^{n}(t) \ge 0\,.  \nonumber
\end{align}
The processes $X^{n}$ can be represented via rate-$1$ Poisson processes:
for each $i \in \cI$ and $t\ge 0$, it holds that
\begin{align} \label{Xrep}
X^{n}_i(t)\;=\;X^{n}_i(0) + A^{n}_i(\lambda^{n} t) - \sum_{j \in \cJ(i)}
S^{n}_{ij} \left( \mu_{ij}^{n}\int_0^t Z_{ij}^{n}(s) \D{s} \right)
- R_i^{n} \left(\gamma_i^{n} \int_0^t Q^{n}_i(s) \D{s} \right)\,,
\end{align}
where the processes $A^{n}_i$, $S^{n}_{ij}$ and $R^{n}_i$ are all rate-1 Poisson
processes and mutually independent, and independent of the initial
quantities $X^{n}_i(0)$. 

\smallskip
\subsubsection{Scheduling control}

We  only consider work conserving policies that are non-anticipative and
preemptive. 
The scheduling decisions are two-fold: (i) when a server becomes free,
if there are customers waiting in one or several buffers,
it has to decide which customer to serve, and (ii) when a customer arrives,
if she finds there are several free servers in one or multiple server pools,
the manager has to decide which server pool to assign the customer to.
These decisions determine the processes $Z^{n}$ at each time. 

Work conservation requires that  whenever there are customers waiting in queues,
if a server becomes free and can serve one of the customers, the server cannot
idle and must decide which customer to serve and start service immediately.
Namely, the processes $Q^{n}$ and $Y^{n}$ satisfy
\begin{equation} \label{wc}
Q_i^{n}(t) \wedge Y^{n}_j(t)\;=\;0
\qquad \forall i \sim j\,, \quad\forall\, t \ge 0\,. 
\end{equation}
Service preemption is allowed, that is, service of a customer can be
interrupted at any time to serve some other customer of another class and
resumed at a later time. 
Following \cite{Atar-05b}, we also consider a stronger condition,
\emph{joint work conservation} (JWC), for preemptive scheduling policies.
Specifically, let $\sX^{n}$ be the set of all possible values of
$X^{n}(t)$ at each time $t\ge 0$ for which there is a rearrangement of customers
such that there is no customer in queue or no idling server in the system and
the processes $Q^{n}$ and $Y^{n}$ satisfy
\begin{equation} \label{jointwc}
e\cdot Q^{n}(t) \wedge e \cdot Y^{n}(t)\;=\;0\,,  \quad t \ge 0 \,. 
\end{equation} 
Note that the set $\sX^{n}$ may not include all possible scenarios
of the system state $X^{n}(t)$ for finite $n$ at each time $t\ge 0$.

We define the action set $\Act^{n}(x)$ as 
\begin{multline*}
\Act^{n}(x)\;\df\; \biggl\{z \in \RR^{I \times J}_+\;\colon\, z_{ij} \le x_i\,,
z_{ij} \le N_j^{n},~ q_i = x_i - \sum_{j \in \cJ(i)} z_{ij}\,,~ y_j = N_j^{n}
- \sum_{i \in \cI(j)}z_{ij}\,, \\[3pt]
\qquad q_i \wedge y_j = 0\quad\forall i \sim j\,,~
e\cdot q \wedge e \cdot y =0  \biggr\}\,. 
\end{multline*}
Then we can write $Z^{n}(t) \in \Act^{n}(X^{n}(t))$ for each $t\ge 0$. 

Define the $\sigma$-fields 
\begin{equation*}
\mathcal{F}^{n}_t \;\df\; \sigma \bigl\{ X^{n}(0), \Tilde{A}^{n}_i(t),
\Tilde{S}^{n}_{ij}(t), \Tilde{R}^{n}_i(t)\;\colon\, i \in \cI, \; j \in \cJ,
\; 0 \le s \le t \bigr\} \vee \mathcal{N} \,,
\end{equation*}
and
\begin{equation*}
\mathcal{G}^{n}_t \;\df\; \sigma \bigl\{ \delta\Tilde{A}^{n}_i(t, r),
\delta \Tilde{S}^{n}_{ij}(t, r), \delta\Tilde{R}^{n}_i(t, r)\;\colon\, i \in \cI,
\; j \in \cJ, \; r \ge 0 \bigr\} \,,
\end{equation*}
where  $\mathcal{N}$ is the collection of all $\Prob$-null sets, 
\begin{equation*}
\Tilde{A}^{n}_i(t) \;\df\; A^{n}_i(\lambda_i^{n} t),
\quad \delta\Tilde{A}^{n}_i(t,r)
\;\df\; \Tilde{A}^{n}_i(t+r) - \Tilde{A}^{n}_i(t) \,, 
\end{equation*}
\begin{equation*}
\Tilde{S}^{n}_{ij}(t) \;\df\; S^{n}_{ij} \left(\mu_{ij}^{n}
\int_0^t Z_{ij}^{n}(s)\,\D{s} \right),
\quad \delta \Tilde{S}^{n}_{ij}(t, r)\;\df\; S^{n}_{ij}
\left( \mu_{ij}^{n}\int_0^t Z_{ij}^{n}(s)\,\D{s} + \mu_{ij}^{n} r \right)
- \Tilde{S}^{n}_{ij}(t) \,,
\end{equation*}
and
\begin{equation*}
\Tilde{R}^{n}_i(t)\;\df\; R_i^{n}
\left(\gamma_i^{n} \int_0^t Q^{n}_i(s)\,\D{s} \right)\,,
\quad \delta \Tilde{R}^{n}_i(t, r)\;\df\; R_i^{n}
\left(\gamma_i^{n} \int_0^t Q^{n}_i(s)\, \D{s}
+ \gamma_i^{n} r \right) -  \Tilde{R}^{n}_i(t) \,. 
\end{equation*}
The filtration $ \mathbf{F}^{n}\df\{\mathcal{F}^{n}_t: t \ge 0\}$ represents
the information available up to time $t$, and the filtration 
$\mathbf{G}^{n}\df\{\mathcal{G}^{n}_t: t \ge 0\}$ contains the information
about future increments of the processes. 
We say that  a scheduling policy is \emph{admissible} if 
\begin{enumerate}
\item[(i)] the `balance' equations in \eqref{baleq} hold.
\smallskip
\item[(ii)] $Z^{n}(t)$ is adapted to $\mathcal{F}^{n}_t$;
\smallskip
\item[(iii)] $\mathcal{F}^{n}_t$ is independent of $\mathcal{G}^{n}_t$ at each time
$t\ge 0$;
\smallskip
\item[(iv)] for each $i \in \cI$ and $i \in \cJ$, and for each
$t\ge 0$, the process  $\delta \Tilde{S}^{n}_{ij}(t, \cdot)$ agrees in law
with $S^{n}_{ij}(\mu_{ij}^{n}\,\cdot)$, and the process
$\delta \Tilde{R}^{n}_i(t, \cdot)$
agrees in law with 
$R^{n}_i (\gamma_i^{n} \cdot)$.  
\end{enumerate}
We denote the set of all admissible scheduling policies
$(Z^{n}, \mathbf{F}^{n}, \mathbf{G}^{n})$ by $\mathfrak{Z}^{n}$.
Abusing the notation we sometimes denote this as $Z^{n}\in\mathfrak{Z}^{n}$.

\smallskip
\subsection{Diffusion Scaling in the Halfin--Whitt regime}
\label{sec-MM-control-hw}
We define the diffusion-scaled processes 
$\Hat{X}^{n}\,=\,(\Hat{X}^{n}_i)_{i\in \cI}\,$,
$\Hat{Q}^{n}\,=\,(\Hat{Q}^{n}_i)_{i\in \cI}\,$,
$\Hat{Y}^{n}\,=\,(\Hat{Y}^{n}_j)_{j\in \cJ}\,$,
and $\Hat{Z}^{n}\,=\,(\Hat{Z}^{n}_{ij})_{i\in \cI,\,j \in \cJ}\,$,  
by
\begin{equation} \label{DiffDef}
\begin{split}
\Hat{X}^{n}_i(t) &\;\df\; \frac{1}{\sqrt{n}} (X_i^{n}(t) - n x_{i}^{*}) \,,  \\
\Hat{Q}^{n}_i(t) &\;\df\; \frac{1}{\sqrt{n}} Q_i^{n}(t) \,,  \\
\Hat{Y}^{n}_j(t) &\;\df\; \frac{1}{\sqrt{n}} Y_j^{n}(t) \,, \\
\Hat{Z}^{n}_{ij}(t) &\;\df\; \frac{1}{\sqrt{n}} (Z_{ij}^{n}(t) - n z_{ij}^*)\,.
\end{split}
\end{equation}

By \eqref{staticfluid}, \eqref{baleq}, and \eqref{DiffDef},
we obtain the balance equations: for  all $t\ge 0$, we have
\begin{equation}\label{baleq-hat}
\begin{split}
& \Hat{X}^{n}_i(t) \;=\; \Hat{Q}_i^{n}(t) + \sum_{j \in \cJ(i)} \Hat{Z}^{n}_{ij}(t)
\qquad \forall i \in \cI\,, \\[5pt]
 & \Hat{Y}^{n}_j(t) + \sum_{i\in \cI(j)} \Hat{Z}^{n}_{ij}(t) \;=\; 0
 \qquad \forall j \in \cJ \,.
\end{split}
\end{equation}
Also, the work conservation conditions in \eqref{wc}, \eqref{jointwc},
translate to
$\Hat{Q}_i^{n}(t) \wedge \Hat{Y}^{n}_j(t)\;=\;0$ for all $i \sim j$,
and $e\cdot \Hat{Q}^{n}(t) \wedge e \cdot \Hat{Y}^{n}(t)\;=\;0$, 
respectively.
By \eqref{baleq-hat},
we obtain
\begin{equation*}
e\cdot \Hat{X}^{n}(t) \;=\; e\cdot \Hat{Q}^{n}(t) - e \cdot \Hat{Y}^{n}(t)\,,
\end{equation*}
and therefore the joint work conservation
condition is equivalent to
\begin{equation}\label{E-joint}
e\cdot \Hat{Q}^{n}(t) \;=\; \bigl(e\cdot \Hat{X}^{n}(t)\bigr)^{+}\,,
\qquad e \cdot \Hat{Y}^{n}(t) \;=\; \bigl(e\cdot \Hat{X}^{n}(t)\bigr)^{-}\,.
\end{equation}
In other words, in the diffusion scale and under joint work conservation,
the total number of customers in queue and the total number of idle servers
are equal to the positive
and negative parts of the centered total number of customers in the system,
respectively.

Let
\begin{equation*}
\begin{split}
\Hat{M}^{n}_{A, i}(t) &\;\df\;  \frac{1}{\sqrt{n}}(A^{n}_i(\lambda_i^{n} t)
- \lambda_i^{n} t), \\[5pt]
 \Hat{M}^{n}_{S, ij}(t) &\;\df\; \frac{1}{\sqrt{n}}\left( S^{n}_{ij}
 \left( \mu_{ij}^{n}\int_0^t Z_{ij}^{n}(s) \D{s} \right)
 - \mu_{ij}^{n}\int_0^t Z_{ij}^{n}(s) \D{s}\right) \,,\\[5pt]
\Hat{M}^{n}_{R, i}(t) &\;\df\;\frac{1}{\sqrt{n}}
\left(R_i^{n} \left(\gamma_i^{n} \int_0^t Q^{n}_i(s) \D{s} \right)
-\gamma_i^{n} \int_0^t Q^{n}_i(s) \D{s} \right)\,.
\end{split}
\end{equation*}
These are square integrable martingales w.r.t. the filtration $\mathbf{F}^{n}$
with quadratic variations
\begin{equation*}
\langle \Hat{M}^{n}_{A, i} \rangle(t) \;\df\; \frac{\lambda_i^{n}}{n} t\,, \quad 
\langle \Hat{M}^{n}_{S, ij} \rangle (t) \;\df\; \frac{\mu_{ij}^{n}}{n}
\int_0^t Z_{ij}^{n}(s)\D{s}\,,\quad 
\langle \Hat{M}^{n}_{R, i} \rangle (t)\;\df\;\frac{\gamma_i^{n}}{n}
\int_0^t Q^{n}_i(s)\D{s} \,. 
\end{equation*}
By \eqref{Xrep}, we can write $\Hat{X}^{n}_i(t)$ as
\begin{multline} \label{hatXn-1}
\Hat{X}^{n}_i(t) \;=\; \Hat{X}^{n}_i(0) + \ell_i^{n} t
- \sum_{j \in \cJ(i)} \mu_{ij}^{n} \int_0^t \Hat{Z}^{n}_{ij}(s) \D{s}
- \gamma^{n}_i \int_0^t \Hat{Q}^{n}_i(s) \D{s}  \\[5pt]
+ \Hat{M}^{n}_{A, i} (t) - \Hat{M}^{n}_{S, ij}(t) - \Hat{M}^{n}_{R, i}(t)\,, 
\end{multline}
where 
$\ell^{n} = (\ell^{n}_1,\dotsc,\ell^{n}_I)\transp$ is defined as
\begin{equation*} 
\ell^{n}_i\;\df\; \frac{1}{\sqrt{n}} \left(\lambda^{n}_i
-  \sum_{i \in \cJ(i)} \mu_{ij}^{n} z_{ij}^* n \right) \,,
\end{equation*}
with $z^*_{ij}$ as defined in \eqref{staticfluid}.
Note that under the assumptions on the parameters in
\eqref{HWpara1}--\eqref{HWpara2} and the first constraint in the LP,
it holds that
\begin{equation} \label{ellconv}
\ell^{n}_i \;\xrightarrow[n\to\infty]{}\;
\ell_i\;\df\; \Hat{\lambda}_i - \sum_{j \in \cJ(i)} \Hat{\mu}_{ij} z_{ij}^*\,. 
\end{equation}
We let $\ell\df(\ell_1,\dotsc,\ell_I)\transp$.

\subsubsection{Control parameterization}\label{S-cparam}
Define the following processes: for $i\in\cI$, and $t\ge 0$, 
\begin{equation} \label{lb-Ucn}
U^{c,n}_i(t) \;\df\; \begin{cases}
\frac{\Hat{Q}^{n}_i(t)}{e\cdot \Hat{Q}^{n}(t)}
& \text{if~} e\cdot \Hat{Q}^{n}(t)>0\\
e_I&\text{otherwise,}\end{cases}
\end{equation} 
and for $j\in\cJ$, and $t\ge 0$,
\begin{equation}\label{lb-Usn}
U^{s,n}_j(t) \;\df\; \begin{cases}
\frac{\Hat{Y}^{n}_j(t)}{e\cdot \Hat{Y}^{n}(t)}
& \text{if~} e\cdot \Hat{Y}^{n}(t)>0\\
e_J&\text{otherwise,}\end{cases}
\end{equation}
The process $U^{c,n}_i(t)$ represents the proportion of the total queue length
in the network at queue $i$ at time $t$, while $U^{s,n}_j(t)$ represents the
proportion of the total idle servers in the network at station $j$ at time $t$.
Let $U^{n}\df (U^{c,n}, U^{s,n})$, with
$U^{c,n}\df (U^{c,n}_1,\dotsc, U^{c,n}_I)\transp$,
and $U^{s,n}\df (U^{s,n}_1,\dotsc, U^{s,n}_J)\transp$.
Given $Z^{n} \in \mathfrak{Z}^{n}$ the process $U^{n}$ is uniquely determined
via \eqref{baleq-hat} and \eqref{lb-Ucn}--\eqref{lb-Usn}
and lives in the set
\begin{equation}\label{E-Act}
\Act \;\df\; \bigl\{ u= (u^c, u^s) \in \RR^{I}_+ \times \RR^J_+
\,\colon\, e\cdot u^c = e\cdot u^s=1\bigr\}\,.
\end{equation}

It follows by \eqref{baleq-hat} and \eqref{E-joint}
that, under the JWC condition, we have that for each $t\ge 0$, 
\begin{equation}\label{lb-joint}
\begin{split}
\Hat{Q}^{n}(t)&\;=\; \bigl(e\cdot \Hat{X}^{n}(t)\bigr)^{+}\,U^{c,n}(t)\,,
\\[5pt]
\Hat{Y}^{n}(t)&\;=\;\bigl(e\cdot \Hat{X}^{n}(t)\bigr)^{-}\,U^{s,n}(t)\,.
\end{split}
\end{equation}

\subsection{The limiting controlled diffusion}
\label{sec-diffusionlimit}

Before introducing the limiting diffusion, we define a mapping to be used for the drift representation as in \cite{Atar-05a, Atar-05b}.
For any $\alpha \in \RR^I$ and $\beta \in \RR^J$, let 
\begin{equation*}
D_G \;\df\; \bigl\{ (\alpha, \beta) \in \RR^{I}\times \RR^{J}\;\colon\,
e \cdot \alpha = e \cdot \beta\bigr\}\,, 
\end{equation*}
and 
define a linear map $G:D_G \to \RR^{I\times J}$ such that 
\begin{equation} \label{diff-map}
\begin{split}
\sum_j \psi_{ij}&\;=\;\alpha_i, \quad \forall i \in \cI\,,  \\
\sum_{i}\psi_{ij}&\;=\;\beta_j, \quad \forall j \in \cJ\,,  \\
\psi_{ij} &\;=\;0\,, \quad \forall i \nsim j\,.
\end{split}
\end{equation}
It is shown in Proposition~A.2 of \cite{Atar-05a} that,
provided $\cG$ is a tree,
there exists a unique map $G$ satisfying \eqref{diff-map}. 
We define the matrix 
\begin{equation} \label{Psi}
\Psi \;\df\; (\psi_{ij})_{i \in \cI,\,j \in \cJ}
= G(\alpha, \beta), \quad \text{for} \quad (\alpha, \beta) \in D_G \,. 
\end{equation}

Following the parameterization in Section~\ref{S-cparam}, we define
the action set $\Act$ as in \eqref{E-Act}.
We use $u^c$ and $u^s$ to represent the control variables for \emph{customer}
classes and \emph{server} pools, respectively, throughout the paper. 
For each $x\in \RR^I$ and $u=(u^c, u^s)\in \Act$, define a mapping 
\begin{equation}\label{diff-map2}
\widehat{G}[u](x) \;\df\; G(x- (e\cdot x)^{+} u^c, - (e\cdot x)^{-} u^s)\,. 
\end{equation}

\begin{remark}\label{R-hG}
The function $\widehat{G}[u](x)$ is clearly well defined for $u=(u^c, u^s) =(0,0)$,
in which case we denote it by $\widehat{G}^{0}(x)$. See also Remark~\ref{R-hG-0}.

\end{remark}

We quote the following result \cite[Lemma~3]{Atar-05b}.

\begin{lemma}\label{L-Xn}
There exists a constant $c_{0}>0$ such that,
whenever $X^{n}\in\Breve{\sX}^{n}$ which is defined by
\begin{equation}\label{E-BsX}
\Breve{\sX}^{n}\;\df\;
\bigl\{x \in \ZZ^{I}_{+}\;\colon\,  \norm{x - n x^*}\le c_{0} n  \bigr\}\,,
\end{equation}
the following holds:
If $Q^{n}\in\ZZ_{+}^{I}$ and $Y^{n}\in\ZZ_{+}^{J}$ satisfy
$(e\cdot Q^{n})\wedge(e\cdot Y^{n})=0$, then
\begin{equation*}
Z^{n}\;=\; G\bigl(X^{n} - Q^{n},N^{n} - Y^{n}\bigr)
\end{equation*}
satisfies $Z^{n}\in\ZZ_{+}^{I\times J}$ and \eqref{baleq} holds.
\end{lemma}

\begin{remark}
It is clear from \eqref{baleq} and \eqref{baleq-hat} that
\begin{align*}
Z^{n}(t) &\;=\; G\bigl(X^{n}(t) - Q^{n}(t), N^{n}- Y^{n}(t)\bigr)\,,\\[5pt]
\Hat{Z}^{n}(t) &\;=\;
G\bigl(\Hat{X}^{n}(t) - \Hat{Q}^{n}(t), - \Hat{Y}^{n}(t)\bigr)\,. 
\end{align*}
Also, by \eqref{lb-joint}, under the JWC condition, we have
$$\Hat{Z}^{n}(t)\;=\;\widehat{G}[U^{n}(t)](\Hat{X}^{n}(t))\,.$$
\end{remark}

Note that the requirement
that $(X^{n} - Q^{n},N^{n} - Y^{n})\in D_{G}$ is an implicit assumption
in the statement of the lemma.
As a consequence of the lemma, $\Breve{\sX}^{n}\subset\sX^{n}$.
Thus, asymptotically as $n \to \infty$,
the JWC condition can be met for all
diffusion scaled system states.

\begin{definition}\label{DEJWC}
We say that $Z^{n}\in\fZ^{n}$ is \emph{jointly work conserving} (JWC) in
a domain $D\subset\RR^{I}$ if \eqref{jointwc} holds
 whenever $\Hat{X}^{n}(t)\in D$.
We say that a sequence $\{Z^{n}\in\fZ^{n},\;n\in\NN\}$ is
\emph{eventually jointly work conserving} (EJWC) if there
is an increasing sequence of domains $D_{n}\subset\RR^{I}$, $n\in\NN$,
which cover $\RR^{I}$ and such that each $Z^{n}$ is JWC on $D_{n}$.
We denote the class of all these sequences by $\boldsymbol\fZ$.
By Lemma~\ref{L-Xn} the class $\boldsymbol\fZ$ is nonempty.
\end{definition}

{\color{black}Under the EJWC condition, the convergence in distribution of the diffusion-scaled processes $\hat{X}^n$ to the
limiting diffusion $X$ in \eqref{hatX} can be proved \cite[Proposition~3]{Atar-05b}. } 

The limit process $X$ is an $I$-dimensional diffusion process,
 satisfying
the It\^o equation 
\begin{equation} \label{hatX}
d X_t\;=\;b(X_t, U_t)\,\D{t} + \Sigma \, \D W_t\,, 
\end{equation}
with initial condition $X_0 = x$ and the control $U_t \in \Act$, 
where the drift $b: \RR^I \times \Act \to \RR^I$ takes the form
\begin{equation} \label{diff-drift}
b_i(x,u)\;=\;b_i(x, (u^c, u^s))\;\df\; - \sum_{j \in \cJ(i)}
\mu_{ij} \widehat{G}_{ij}[u](x)
- \gamma_i (e\cdot x)^{+} u^c_i + \ell_i \qquad\forall\,i\in\cI\,,
\end{equation}
and the covariance matrix is given by 
\begin{equation*}
\Sigma\;\df\;\diag\bigl(\sqrt{2 \lambda_1}, \dotsc, \sqrt{2 \lambda_I}\bigr)\,. 
\end{equation*}
Let $\Uadm$ be the set of all admissible controls for the limiting diffusion
(see Section~\ref{sec-general}). 

The limiting processes $Q$, $Y$, and $Z$ satisfy the following:
$Q_i \ge 0$ for $i \in \cI$, $Y_j \ge 0$ for $j \in \cJ$, and 
for all $t\ge 0$, and it holds that
\begin{equation}\label{baleq-diff}
\begin{split}
& X_i(t) \;=\; Q_i(t) + \sum_{j \in \cJ(i)} Z_{ij}(t)
\quad \forall i \in \cI\,,  \\[5pt]
 & Y_j(t) + \sum_{i\in \cI(j)} Z_{ij}(t) \;=\; 0
 \quad \forall j \in \cJ \,.
\end{split}
\end{equation}
Note that these `balance' conditions
imply that joint work conservation always holds at the diffusion limit, i.e.,
\begin{equation}\label{E-joint-diff}
e\cdot Q(t) \;=\; \bigl(e\cdot X(t)\bigr)^{+}\,,
\qquad e \cdot Y(t) \;=\; \bigl(e\cdot X(t)\bigr)^{-} \qquad \forall\,t \ge 0 \,. 
\end{equation}
It is clear then that by \eqref{diff-map} and \eqref{E-joint-diff}, we have
$$Z(t) \;=\; G\bigl(X(t)-Q(t),-Y(t)\bigr)\,.$$

\subsection{The limiting diffusion ergodic control problems}
\label{sec-MM-control-diff}


We now introduce two formulations of ergodic control problems for the
limiting diffusion. 

(1) \emph{Unconstrained control problem.}
Define the running cost function $r: \RR^I\times \Act \to \RR^I$ by
\begin{equation*}
r(x, u)\;=\;r(x,(u^c, u^s))\,,
\end{equation*}
where  
\begin{equation}\label{E-cost}
r(x, u) \;=\; [(e\cdot x)^{+}]^m \sum_{i=1}^{I} \xi_i (u^c_i)^m
+  [(e\cdot x)^{-}]^m \sum_{j=1}^{J} \zeta_j (u^s_j)^m, \quad m\ge 1\,,
\end{equation}
for some positive vectors $\xi=(\xi_1,\dotsc,\xi_I)\transp$ 
and $\zeta=(\zeta_1,\dotsc,\zeta_J)\transp$. 

The ergodic criterion associated with the controlled diffusion $X$ and the
running cost $r$ is defined as
\begin{equation}\label{diff-cost}
J_{x,U}[r] \;\df\; \limsup_{T \to \infty}\;\frac{1}{T}\;\Exp_x^U
\left[ \int_{0}^{T} r(X_t, U_t)\,\D{t} \right]\,, \quad U \in \Uadm\,. 
\end{equation}
The ergodic cost minimization problem is then defined as
\begin{equation}\label{diff-opt}
\varrho^*(x) \;=\; \inf_{U \in \Uadm} \;J_{x,U}[r] \,.
\end{equation}
The quantity $\varrho^*(x)$ is called the optimal value of
the ergodic control problem 
for the controlled diffusion process $X$ with initial state $x$. 

(2) \emph{Constrained control problem.} 
The  second formulation of the ergodic control problem
 is as follows.
The running cost function $r_{0}(x,u)$ is as defined in
\eqref{E-cost} with $\zeta\equiv 0$.
Also define
\begin{equation}\label{E-rj}
r_j(x,u)\;\df\; [(e\cdot x)^{-}u^s_j]^m\,,\quad j\in\cJ\,,
\end{equation}
and let $\updelta=(\updelta_1,\dotsc,\updelta_J)$ be a positive vector.
The ergodic cost minimization problem under idleness constraints is defined as 
\begin{align} 
\varrho_{c}^{*}(x) &\;=\; \inf_{U \in \Uadm} \;J_{x,U}[r_{0}]\label{diff-opt-c1}
\\[5pt]
\text{subject to}\quad
J_{x,U}[r_{j}] &\;\le\; \updelta_j\,,\quad j \in \cJ\,.\label{diff-opt-c2} 
\end{align}
The constraint in \eqref{diff-opt-c2} can be written as
\begin{equation*}
\lim_{T\to\infty}\;\frac{1}{T}\;\Exp_x^U
\Biggl[\int_{0}^{T} \Biggl(-\sum_{i \in \cI(j)}
\widehat{G}_{ij}[U](X_t) \Biggr)^m\,\D{t} \Biggr] \;\le\; \updelta_j\,,
\quad j \in \cJ\,.
\end{equation*}

As we show in Section~\ref{S3}, the optimal values $\varrho^*(x)$ and
$\varrho^*_c(x)$ do not depend on $x\in\RR^{I}$,
and thus we remove their dependence on $x$ in the statements below.
We  prove the well-posedness of these ergodic diffusion control problems,
and characterize their optimal solutions in Sections~\ref{sec-stabilizability}
and \ref{sec-HJB}.

\section{Ergodic Control of a Broad Class of Controlled Diffusions}
\label{S3}

We review the model and the structural properties of a broad class of
controlled diffusions for which the ergodic control problem is
well posed \cite{ABP14}.
We augment the results in \cite{ABP14} with the study of ergodic control
under constraints.

\subsection{The model} \label{sec-general}
Consider a controlled diffusion process $X = \{X_{t},\;t\ge0\}$
taking values in the $d$-dimensional Euclidean space $\RR^{d}$, and
governed by the It\^o stochastic differential equation
\begin{equation}\label{E-sde}
\D{X}_{t} \;=\;b(X_{t},U_{t})\,\D{t} + \upsigma(X_{t})\,\D{W}_{t}\,.
\end{equation}
All random processes in \eqref{E-sde} live in a complete
probability space $(\Omega,\sF,\Prob)$.
The process $W$ is a $d$-dimensional standard Wiener process independent
of the initial condition $X_{0}$.
The control process $U$ takes values in a compact, metrizable set $\Act$, and
$U_{t}(\omega)$ is jointly measurable in
$(t,\omega)\in[0,\infty)\times\Omega$.
Moreover, it is \emph{non-anticipative}:
for $s < t$, $W_{t} - W_{s}$ is independent of
\begin{equation*}
\sF_{s} \;\df\;\text{the completion of~} \sigma\{X_{0},U_{r},W_{r},\;r\le s\}
\text{~relative to~}(\sF,\Prob)\,.
\end{equation*}
Such a process $U$ is called an \emph{admissible control}.
Let $\Uadm$ denote the set of all admissible controls.

We impose the following standard assumptions on the drift $b$
and the diffusion matrix $\upsigma$
to guarantee existence and uniqueness of solutions to equation \eqref{E-sde}.
\begin{itemize}
\item[{(A1)}]
\emph{Local Lipschitz continuity:\/}
The functions
\begin{equation*}
b\;=\;\bigl[b_{1},\dotsc,b_{d}\bigr]\transp\,\colon\,\RR^{d}\times\Act\to\RR^{d}\,,
\quad\text{and}\quad
\upsigma\;=\;\bigl[\upsigma_{ij}\bigr]\,\colon\,\RR^{d}\to\RR^{d\times d}
\end{equation*}
are locally Lipschitz in $x$ with a Lipschitz constant $C_{R}>0$ depending on
$R>0$.
In other words,
for all $x,y\in B_{R}$ and $u\in\Act$,
\begin{equation*}
\abs{b(x,u) - b(y,u)} + \norm{\upsigma(x) - \upsigma(y)}
\;\le\;C_{R}\,\abs{x-y}\,.
\end{equation*}
We also assume that $b$ is continuous in $(x,u)$.
\smallskip
\item[{(A2)}]
\emph{Affine growth condition:\/}
$b$ and $\upsigma$ satisfy a global growth condition of the form
\begin{equation*}
\abs{b(x,u)}^{2}+ \norm{\upsigma(x)}^{2}\;\le\;C_{1}
\bigl(1 + \abs{x}^{2}\bigr) \qquad \forall (x,u)\in\RR^{d}\times\Act\,,
\end{equation*}
where $\norm{\upsigma}^{2}\;\df\;
\mathrm{trace}\left(\upsigma\upsigma\transp\right)$.
\smallskip
\item[{(A3)}]
\emph{Local nondegeneracy:\/}
For each $R>0$, it holds that
\begin{equation*}
\sum_{i,j=1}^{d} a_{ij}(x)\xi_{i}\xi_{j}
\;\ge\;C^{-1}_{R}\abs{\xi}^{2} \qquad\forall x\in B_{R}\,,
\end{equation*}
for all $\xi=(\xi_{1},\dotsc,\xi_{d})\transp\in\RR^{d}$,
where $a\df \upsigma \upsigma\transp$.
\end{itemize}

In integral form, \eqref{E-sde} is written as
\begin{equation}\label{E3.2}
X_{t} \;=\;X_{0} + \int_{0}^{t} b(X_{s},U_{s})\,\D{s}
+ \int_{0}^{t} \upsigma(X_{s})\,\D{W}_{s}\,.
\end{equation}
The third term on the right hand side of \eqref{E3.2} is an It\^o
stochastic integral.
We say that a process $X=\{X_{t}(\omega)\}$ is a solution of \eqref{E-sde},
if it is $\sF_{t}$-adapted, continuous in $t$, defined for all
$\omega\in\Omega$ and $t\in[0,\infty)$, and satisfies \eqref{E3.2} for
all $t\in[0,\infty)$ a.s.
It is well known that under (A1)--(A3), for any admissible control
there exists a unique solution of \eqref{E-sde}
\cite[Theorem~2.2.4]{book}.

The \emph{controlled extended generator} $\Lg^{u}$ of
the diffusion is defined by
$\Lg^{u}\colon\Cc^{2}(\RR^{d})\to\Cc(\RR^{d})$,
where $u\in\Act$ plays the role of a parameter, by
\begin{equation}\label{E3.3}
\Lg^{u} f(x) \;\df\;\frac{1}{2} \sum_{i,j=1}^{d} a_{ij}(x)\,\partial_{ij} f(x)
+ \sum_{i=1}^{d} b_{i}(x,u)\, \partial_{i} f(x)\,,\quad u\in\Act\,.
\end{equation}
We adopt
the notation $\partial_{i}\df\tfrac{\partial~}{\partial{x}_{i}}$ and
$\partial_{ij}\df\tfrac{\partial^{2}~}{\partial{x}_{i}\partial{x}_{j}}$.

Of fundamental importance in the study of functionals of $X$ is
It\^o's formula.
For $f\in\Cc^{2}(\RR^{d})$ and with $\Lg^{u}$ as defined in \eqref{E3.3},
it holds that
\begin{equation}\label{E3.4}
f(X_{t}) \;=\;f(X_{0}) + \int_{0}^{t}\Lg^{U_{s}} f(X_{s})\,\D{s}
+ M_{t}\,,\quad\text{a.s.},
\end{equation}
where
\begin{equation*}
M_{t} \;\df\;\int_{0}^{t}\bigl\langle\nabla f(X_{s}),
\upsigma(X_{s})\,\D{W}_{s}\bigr\rangle
\end{equation*}
is a local martingale.
Krylov's extension of It\^o's formula \cite[p.~122]{krylov}
extends \eqref{E3.4} to functions $f$ in the local
Sobolev space $\Sobl^{2,p}(\RR^{d})$, $p\ge d$.

Recall that a control is called \emph{Markov} if
$U_{t} = v(t,X_{t})$ for a measurable map $v\colon\RR_{+}\times\RR^{d}\to \Act$,
and it is called \emph{stationary Markov} if $v$ does not depend on
$t$, i.e., $v\colon\RR^{d}\to \Act$.
Correspondingly \eqref{E-sde}
is said to have a \emph{strong solution}
if given a Wiener process $(W_{t},\sF_{t})$
on a complete probability space $(\Omega,\sF,\Prob)$, there
exists a process $X$ on $(\Omega,\sF,\Prob)$, with $X_{0}=x_{0}\in\RR^{d}$,
which is continuous,
$\sF_{t}$-adapted, and satisfies \eqref{E3.2} for all $t$ a.s.
A strong solution is called \emph{unique},
if any two such solutions $X$ and $X'$ agree
$\Prob$-a.s., when viewed as elements of $\Cc\bigl([0,\infty),\RR^{d}\bigr)$.
It is well known that under Assumptions (A1)--(A3),
for any Markov control $v$,
\eqref{E-sde} has a unique strong solution \cite{Gyongy-96}.

Let $\Usm$ denote the set of stationary Markov controls.
Under $v\in\Usm$, the process $X$ is strong Markov,
and we denote its transition function by $P^{t}_{v}(x,\cdot\,)$.
It also follows from the work of \cite{Bogachev-01,Stannat-99} that under
$v\in\Usm$, the transition probabilities of $X$
have densities which are locally H\"older continuous.
Thus $\Lg^{v}$ defined by
\begin{equation*}
\Lg^{v} f(x) \;\df\;\frac{1}{2}\sum_{i,j=1}^{d} a_{ij}(x)\,\partial_{ij} f(x)
+\sum_{i=1}^{d} b_{i} \bigl(x,v(x)\bigr)\,\partial_{i} f(x)\,,\quad v\in\Usm\,,
\end{equation*}
for $f\in\Cc^{2}(\RR^{d})$,
is the generator of a strongly-continuous
semigroup on $\Cc_{b}(\RR^{d})$, which is strong Feller.
We let $\Prob_{x}^{v}$ denote the probability measure and
$\Exp^{v}_{x}$ the expectation operator on the canonical space of the
process under the control $v\in\Usm$, conditioned on the
process $X$ starting from $x\in\RR^{d}$ at $t=0$.

Recall that control $v\in\Usm$ is called \emph{stable}
if the associated diffusion is positive recurrent.
We denote the set of such controls by $\Ussm$,
and let $\mu_{v}$ denote the unique invariant probability
measure on $\RR^{d}$ for the diffusion under the control $v\in\Ussm$.
We also let $\cM\df\{\mu_{v}\,\colon\,v\in\Ussm\}$, and
$\eom$ denote the set of ergodic occupation measures corresponding to controls
in $\Ussm$, that is, 
\begin{equation*}\eom\;\df\;\biggl\{\uppi\in\cP(\RR^{d}\times\Act)\,\colon\,
\int_{\RR^{d}\times\Act}\Lg^{u} f(x)\,\uppi(\D{x},\D{u})=0\quad
\forall\,f\in\Cc^{\infty}_c(\RR^{d}) \biggr\}\,,\end{equation*}
where $\Lg^{u}f(x)$ is given by \eqref{E3.3}.

We need the following definition:

\begin{definition}
A function $h\colon\RR^{d}\times\Act\to\RR$ is called \emph{inf-compact}
on a set $A\subset \RR^{d}$ if the set
$\Bar{A}\cap\bigl\{x\,\colon\,\min_{u\in\Act}\;h(x,u)\le c\bigr\}$
is compact (or empty) in $\RR^{d}$ for all $c\in\RR$.
When this property
holds for $A\equiv \Rd$, then we simply say that $h$ is inf-compact.
\end{definition}

Recall that $v\in\Ussm$ if and only if there exists an inf-compact function
$\Lyap\in\Cc^{2}(\RR^{d})$, a bounded domain $D\subset\RR^{d}$, and
a constant $\varepsilon>0$ satisfying
\begin{equation*}
\Lg^{v}\Lyap(x) \;\le\;-\varepsilon
\qquad\forall x\in D^{c}\,.
\end{equation*}
We denote by $\uptau(A)$ the \emph{first exit time} of a process
$\{X_{t}\,,\;t\in\RR_{+}\}$ from a set $A\subset\RR^{d}$, defined by
\begin{equation*}
\uptau(A) \;\df\;\inf\;\{t>0\,\colon\,X_{t}\not\in A\}\,.
\end{equation*}
The open ball of radius $R$ in $\RR^{d}$, centered at the origin,
is denoted by $B_{R}$, and we let $\uptau_{R}\;\df\;\uptau(B_{R})$,
and $\tc_{R}\df \uptau(B^{c}_{R})$.

We assume that the running cost function $r(x,u)$
is nonnegative, continuous and locally Lipschitz
in its first argument uniformly in $u\in\Act$.
Without loss of generality we let $C_{R}$
be a Lipschitz constant of $r(\,\cdot\,,u)$ over $B_{R}$.
In summary, we assume that
\begin{itemize}
\item[{(A4)}]
$r\colon\RR^{d}\times\Act\to\RR_{+}$ is continuous and satisfies,
for some constant $C_{R}>0$
\begin{equation*}
\babs{r(x,u)-r(y,u)} \;\le\;C_{R}\,\abs{x-y}
\qquad\forall x,y\in B_{R}\,,~\forall u\in\Act\,,
\end{equation*}
and all $R>0$.
\end{itemize}

In general, $\Act$ may not be a convex set.
It is therefore often useful to enlarge the control set to $\cP(\Act)$.
For any $v(\D{u})\in\cP(\Act)$ we can redefine the drift and the 
running cost as
\begin{equation}\label{relax}
\Bar{b}(x,v)\;\df\;\int_\Act b(x,u)v(\D{u})\,,\quad \text{and}\quad
\Bar{r}(x,v)\;\df\;\int_\Act r(x,u)v(\D{u})\,.
\end{equation}
It is easy to see that the drift and running cost
defined in \eqref{relax} satisfy all the
aforementioned conditions {(A1)}--{(A4)}.
In what follows we assume that
all the controls take values in $\cP(\Act)$.
These controls are generally referred to as \emph{relaxed} controls,
while a control taking values in $\Act$ is called \emph{precise}.
We endow the set of relaxed stationary Markov controls with the following
topology: $v_{n}\to v$ in $\Usm$ if and only if
\begin{equation*}
\int_{\RR^{d}}f(x)\int_\Act g(x,u)v_{n}(\D{u}\mid x)\,\D{x}
\;\xrightarrow[n\to\infty]{}\;
\int_{\RR^{d}}f(x)\int_\Act g(x,u)v(\D{u}\mid x)\,\D{x}
\end{equation*}
for all $f\in L^{1}(\RR^{d})\cap L^{2}(\RR^{d})$
and $g\in\Cc_{b}(\RR^{d}\times \Act)$.
Then $\Usm$ is a compact metric space under this topology
\cite[Section~2.4]{book}.
We refer to this topology as the \emph{topology of Markov controls}.
A control is said to be \emph{precise} if it takes value in $\Act$.
It is easy to see that any precise control $U_{t}$ can also be understood
as a relaxed control by $U_{t}(\D{u})=\delta_{U_{t}}$.
Abusing the notation we denote the drift and running cost by $b$ and $r$,
respectively, and the action of a relaxed control
on them is understood as in \eqref{relax}.
In this manner, the
definition of $J_{x,U}[r]$ in
\eqref{diff-cost}, is naturally extended to relaxed $U\in\Uadm$ and $x\in \RR^d$.
For $v\in\Ussm$, the functional $J_{x,v}[r]$ does not depend on $x\in \RR^d$.
In this case we drop the dependence on $x$ and
denote this by $J_{v}[r]$.
Note that if $\uppi_{v}(\D{x},\D{u})\df\mu_{v}(\D{x})\,v(\D{u}\mid x)$
is the ergodic occupation measure corresponding to $v\in\Ussm$,
then we have
\begin{equation*}
J_{v}[r] \;=\; \int_{\Rd\times\Act} r(x,u)\,\uppi_{v}(\D{x},\D{u})\,.
\end{equation*}
Therefore, the restriction of the ergodic control problem in
\eqref{diff-opt} to stable stationary Markov controls is equivalent
to minimizing
\begin{equation*}
\uppi(r)\;=\;\int_{\Rd\times\Act} r(x,u)\,\uppi(\D{x},\D{u})
\end{equation*}
over all $\uppi\in\eom$.
If the infimum is attained in $\eom$, then we say that the ergodic control
problem is well posed, and we refer to
any $\Bar\uppi\in\eom$ that attains this infimum
as an \emph{optimal ergodic occupation measure}.

\subsection{Hypotheses and review
of some results from \cite{ABP14}} \label{sec-structural}

A structural hypothesis was introduced in \cite{ABP14}
to study ergodic control
for a broad class of controlled diffusion models. 
This is as follows:

\begin{hypothesis}\label{HypA}
For some open set $\cK\subset\RR^{d}$,
the following hold:
\begin{itemize}
\item[(i)]
The running cost $r$ is inf-compact on $\cK$.
\smallskip
\item[(ii)]
There exist inf-compact functions $\Lyap\in\Cc^{2}(\RR^{d})$
and  $h \in \Cc(\RR^{d}\times\Act)$, such that
\begin{equation*}
\begin{split}
\Lg^{u}\Lyap(x) &\;\le\;1 - h(x,u)\qquad
\forall\,(x,u)\in \cK^{c}\times\Act\,,\\[5pt]
\Lg^{u}\Lyap(x) &\;\le\;1 + r(x,u)
\qquad\forall\,(x,u)\in\cK\times\Act\,.
\end{split}
\end{equation*}
\end{itemize}
Without loss of generality, we assume that $\Lyap$ and $h$ are nonnegative. 
\end{hypothesis}

In Hypothesis~\ref{HypA},
for notational economy, and without loss of generality, we refrain from using
any constants.
Observe that for $\cK=\RR^{d}$ the problem reduces to an
ergodic control problem with inf-compact
cost, and for $\cK=\varnothing$ we obtain an ergodic control problem for a
uniformly stable controlled diffusion.
As shown in \cite{ABP14}, Hypothesis~\ref{HypA} implies that
\begin{equation*}
J_{x,U}\bigl[h\,\Ind_{\cK^{c}\times\Act}\bigr]
\;\le\;J_{x,U}\bigl[r\,\Ind_{\cK\times\Act}\bigr]\qquad\forall\,U\in\Uadm\,.
\end{equation*} 

The hypothesis that follows is necessary for the value of the ergodic control
problem to be finite.
It is a standard assumption in ergodic control.

\begin{hypothesis}\label{HypB}
There exists $\Hat{U}\in\Uadm$ such that $J_{x,\Hat{U}}[r]<\infty$
for some $x\in\RR^{d}$.
\end{hypothesis}

It is shown in \cite{ABP14} that under Hypotheses~\ref{HypA} and \ref{HypB}
the ergodic control problem in \eqref{diff-cost}--\eqref{diff-opt} is
well posed.
The following result which is contained in Lemma~3.3 and Theorem~3.1
of \cite{ABP14}
plays a key role in the analysis of the problem.
Let
\begin{equation*}
\cH\;\df\; (\cK\times\Act) \;\textstyle{\bigcup}\;
\bigl\{(x,u)\in\RR^{d}\times\Act\,\colon\,
r(x,u)> h(x,u)\bigr\}\,,
\end{equation*}
where $\cK$ is the open set in Hypothesis~\ref{HypA}.

\begin{lemma}\label{L3.1}
Under Hypothesis~\ref{HypA}, the following are true.
\begin{itemize}
\item[(a)]
There exists an inf-compact function
$\Tilde{h}\in\Cc(\RR^{d}\times\Act)$
which is locally Lipschitz in its first argument uniformly w.r.t. its
second argument, and satisfies
\begin{equation}\label{E-Lkey}
r(x,u) \;\le\;\Tilde{h}(x,u) \;\le\;\frac{k_{0}}{2}\,
\bigl(1+h(x,u)\,\Ind_{\cH^{c}}(x,u)+r(x,u)\,\Ind_{\cH}(x,u)\bigr)
\end{equation}
for all $(x,u)\in\RR^{d}\times\Act$,
and for some positive constant $k_{0}\ge2$.
\smallskip
\item[(b)]
The function $\Lyap$ in Hypothesis~\ref{HypA} satisfies
\begin{equation*}
\Lg^{u}\Lyap(x) \;\le\;1 - h(x,u)\,\Ind_{\cH^{c}}(x,u)+r(x,u)\,\Ind_{\cH}(x,u)
\qquad \forall (x,u)\in\RR^{d}\times\Act\,.
\end{equation*}
\item[(c)]
It holds that
\begin{equation}\label{E-key}
J_{x,U}\bigl[\Tilde{h}\bigr]
\;\le\;k_{0}\bigl(1+J_{x,U}[r]\bigr)\qquad\forall\,U\in\Uadm\,.
\end{equation}
\end{itemize}
\end{lemma}

Hypothesis~\ref{HypB} together with \eqref{E-key} imply that
$J_{x,\Hat{U}}\bigl[\Tilde{h}\bigr]<\infty$.
This together with the fact that $\Tilde{h}$
is inf-compact and dominates $r$ is used in \cite{ABP14}
to prove that the ergodic control problem is well posed.
Also, there exists a constant $\varrho^{*}$ such that
\begin{equation}\label{E-average}
\varrho^{*} \;=\; \inf_{U\in\Uadm}\;\limsup_{T \to \infty}\;\frac{1}{T}\;\Exp_x^U
\left[ \int_{0}^{T} r(X_t, U_t)\,\D{t} \right]\,, \quad \forall\, x\in\Rd\,. 
\end{equation}
Moreover, the infimum in \eqref{E-average} is attained at a precise
stationary Markov control, and the set of
optimal stationary Markov controls is characterized
via a HJB equation that has a unique solution in a certain class
of functions \cite[Theorems~3.4 and 3.5]{ABP14}.

Another important result in \cite{ABP14} is an approximation
technique which plays a crucial role in
the proof of asymptotic optimality (as $n\to\infty$)
of the Markov control obtained from the HJB
for the ergodic control problem of the multiclass single-pool queueing systems.
In summary this can be described as follows.
We truncate the data of the problem by fixing the control
outside a ball in $\Rd$.
The control is chosen in a manner that the set of ergodic occupation
measures for the truncated problem is compact.
We have shown that as the radius of the ball tends to infinity,
 the optimal value of the truncated problem converges to
the optimal value of the original problem.

The precise definition of the `truncated' model is as follows.

\begin{definition}\label{D-truncation}
Let $v_{0}\in\Ussm$ be any control such that $\uppi_{v_{0}}(r)<\infty$.
We fix the control $v_{0}$ on the complement of the ball $\Bar{B}_{R}$
and leave the parameter $u$ free inside.
In other words, for each $R\in\NN$ we define
\begin{align}
b^{R}(x,u) &\;\df\;\begin{cases} b(x,u)&\text{if~~}
(x,u)\in \Bar{B}_{R}\times\Act\,,\\[2pt]
b(x,v_{0}(x))& \text{otherwise,}\end{cases} \label{bR}\\[5pt]
r^{R}(x,u) &\;\df\;\begin{cases} r(x,u)&\text{if~~}
(x,u)\in \Bar{B}_{R}\times\Act\,,\\[2pt]
r(x,v_{0}(x))&\text{otherwise.}\end{cases} \label{rR}
\end{align}
Consider the ergodic control problem
for the family of controlled diffusions, parameterized by $R\in\NN$,
given by
\begin{equation}\label{E-sdeR}
\D{X}_{t} \;=\;b^{R}(X_{t},U_{t})\,\D{t} + \upsigma(X_{t})\,\D{W}_{t}\,,
\end{equation}
with associated running costs $r^{R}(x,u)$.
We denote by $\Usm(R,v_{0})$ the subset of $\Usm$ consisting of those controls
$v$ which agree with $v_{0}$ on $\Bar{B}_{R}^{c}$,
and by $\eom(R)$ we denote the set of ergodic occupation
measures of \eqref{E-sdeR}.
\end{definition}

Let $\eta_{0}\df \uppi_{v_{0}}(\Tilde{h})$.
By \eqref{E-key}, $\eta_{0}$ is finite.
Let $\varphi_{0}\in\Sobl^{2,p}(\RR^{d})$, for any $p>d$, be the
minimal nonnegative solution
to the Poisson equation
(see \cite[Lemma~3.7.8\,(ii)]{book})
\begin{equation}\label{E-vphi0}
\Lg^{v_{0}}\varphi_{0}(x) \;=\; \eta_{0} - \Tilde{h}(x,v_{0}(x))\,\quad
x\in\RR^{d}\,.
\end{equation}
Under Hypotheses~\ref{HypA} and \ref{HypB}, all the conclusions of
Theorems~4.1 and 4.2 in \cite{ABP14} hold.
Consequently, we have the following lemma.

\begin{lemma}\label{L3.2}
Under Hypotheses~\ref{HypA} and \ref{HypB}, the following hold.
\begin{enumerate}
\item[(i)]
The set $\eom(R)$ is compact for each $R>0$, and thus
the set of optimal ergodic occupation measures for $r^R$ in $\eom(R)$, denoted as
$\Bar\eom(R)$, is nonempty.
\smallskip
\item[(ii)]
The collection $\cup_{R>0}\,\Bar\eom(R)$ is tight in $\cP(\Rd\times\Act)$.
\end{enumerate}
Moreover, provided
$\varphi_{0}\in\order\bigl(\min_{u\in\Act}\;\Tilde{h}(\cdot\,,u)\bigr)$,
 for any collection
$\{\Bar\uppi^{R}\in\Bar\eom(R)\;\colon\,R>0\}$, we have
\begin{itemize}
\item[(iii)]
Any limit point of $\Bar\uppi^{R}$ as $R\to\infty$
is an optimal ergodic occupation
measure of \eqref{E-sde} for $r$.
\smallskip
\item[(iv)]
It holds that $\lim_{R\nearrow\infty}\; \Bar\uppi^{R}\bigl(r^{R}\bigr)
= \varrho^{*}$.
\end{itemize}
\end{lemma}

\subsection{Ergodic control under constraints} \label{sec-constraint}

Let $r_{i}\,\colon\,\RR^{d}\to\RR_{+}$, $0\le i\le \Bar{k}$,
be a set of continuous functions, each satisfying (A4).  Define 
\begin{equation} \label{E-ec-r}
r \;\df \; \sum_{i=0}^{\Bar{k}} r_{i}\,.
\end{equation}
We are also given a set of positive constants $\updelta_{i}$,
$i=1,\dotsc,\Bar{k}$.
The objective is to minimize
\begin{equation} \label{E-ec-c1}
\uppi(r_{0})\;=\;\int_{\Rd\times\Act} r_{0}(x,u)\,\uppi(\D{x},\D{u})
\end{equation}
over all $\uppi\in\eom$, subject to
\begin{equation} \label{E-ec-c2}
\uppi(r_{i})\;=\;\int_{\Rd\times\Act} r_{i}(x,u)\,\uppi(\D{x},\D{u})
\;\le\;\updelta_{i}\,,
\qquad i=1,\dotsc,\Bar{k}\,.
\end{equation}
For
$\updelta=(\updelta_1,\dotsc,\updelta_{\Bar{k}})\in\RR_{+}^{\Bar{k}}$
let
\begin{equation}\label{E-sH}
\begin{split}
\sH(\updelta)
\;\df\; \bigl\{\uppi\in\eom\;\colon\, \uppi(r_{i})\le \updelta_{i}\,,\;
i=1,\dotsc,\Bar{k}\}\,,\\[5pt]
\sH^{\mathrm{o}}(\updelta)
\;\df\; \bigl\{\uppi\in\eom\;\colon\,
\uppi(r_{i})< \updelta_{i}\,,\;
i=1,\dotsc,\Bar{k}\}\,.
\end{split}
\end{equation}
It is straightforward to show that $\sH(\updelta)$ is convex and closed in $\eom$.
Let $\sH_{\mathrm{e}}(\updelta)$
($\eom_{\mathrm{e}}$) denote the set of extreme points
of $\sH(\updelta)$ ($\eom$).

Throughout this section we assume that Hypothesis~\ref{HypA}
holds for $r$ in \eqref{E-ec-r} without any further mention.
We have the following lemma.

\begin{lemma}\label{L3.3}
Suppose that
\begin{equation*}
\sH(\updelta)\cap\{\uppi\in\eom\;\colon\, \uppi(r_{0})<\infty\}\ne\varnothing\,.
\end{equation*}
Then there exists $\uppi^{*}\in\sH(\updelta)$ such that
\begin{equation*}
\uppi^{*}(r_{0}) \;=\; \inf_{\uppi\,\in\,\sH(\updelta)}\; \uppi(r_{0})\,.
\end{equation*}
Moreover, $\uppi^{*}$ may be selected so as to correspond to a
precise stationary Markov control.
\end{lemma}

\begin{proof}
By hypothesis, there exists $\updelta_{0}\in\RR_+$ such that
$\widehat\sH \df\sH(\updelta) \cap 
\{\uppi\in\eom\,\colon\,\uppi(r_{0})\le\updelta_{0}\}\ne\varnothing$.
By \eqref{E-key} we have
\begin{equation}\label{EL3.3b}
\uppi(\Tilde{h})\;\le\; k_{0}+ k_{0}\,\sum_{i=1}^{\Bar{k}}\updelta_{i}
+k_{0}\,\uppi(r_{0})
\qquad\forall\, \uppi\in\sH(\updelta)\,,
\end{equation}
which implies, since $\Tilde{h}$ is inf-compact, that
$\widehat\sH$ is pre-compact in $\cP(\Rd\times\Act)$.
Let $\uppi_{n}$ be any sequence in $\widehat\sH$ such that
\begin{equation*}
\uppi_{n}(r_{0})\;\xrightarrow[n\to\infty]{} \varrho_{0}\;\df\;
\inf_{\uppi\,\in\,\sH(\updelta)}\; \uppi(r_{0})\,.
\end{equation*}
By compactness $\uppi_{n}\to\uppi^{*}\in\cP(\Rd\times\Act)$ along some subsequence.
Since $\eom$ is closed in $\cP(\Rd\times\Act)$,
it follows that $\uppi^{*}\in\eom$.
On the other hand, since the functions $r_{i}$ are continuous and bounded below,
it follows that the map $\uppi\mapsto\uppi(r_{i})$ is lower-semicontinuous,
which implies that
$\uppi^{*}(r_{0})\le \varrho_{0}$
and $\uppi^{*}(r_{i})\le \updelta_{i}$ for $i=1,\dotsc,\Bar{k}$.
It follows that $\uppi^{*}\in\widehat\sH\subset\sH(\updelta)$.
Therefore, $\widehat\sH$ is closed, and therefore also compact.

Applying Choquet's theorem as in the proof of \cite[Lemma~4.2.3]{book},
it follows that there exists
$\Tilde\uppi^{*}\in\widehat\sH_{\mathrm{e}}$, the set of
extreme points of $\widehat\sH$, such that $\Tilde\uppi^{*}(r_{0})=\varrho_{0}$.
On the other hand, we have
$\widehat\sH_{\mathrm{e}}\subset\eom_{\mathrm{e}}$ by \cite[Lemma~4.2.5]{book}.
It follows that $\Tilde\uppi^{*}\in\sH(\updelta)\cap\eom_{\mathrm{e}}$.
Since every element of $\eom_{\mathrm{e}}$ corresponds to a precise
stationary Markov control, the proof is complete.
\end{proof}

\begin{definition}
We say that the vector $\updelta\in (0,\infty)^{\Bar{k}}$
is \emph{feasible} (or that the constraints in \eqref{E-ec-c2} are feasible)
if  there exists $\uppi'\in\sH^{\mathrm{o}}(\updelta)$ such that
$\uppi'(r_{0})<\infty$.
\end{definition}

\begin{lemma}\label{L3.4}
If $\Hat\updelta\in (0,\infty)^{\Bar{k}}$ is feasible,  then
$\updelta\mapsto\inf_{\uppi\,\in\,\sH(\updelta)}\;\uppi(r_{0})$
is continuous at $\Hat\updelta$.
\end{lemma}

\begin{proof}
This follows directly from the fact that, since
$\Hat\updelta$ is feasible, the primal functional
\begin{equation*}
\updelta\,\mapsto\, \inf_{\uppi\in\eom}\,\{\uppi(r_{0})\,\colon\,
\uppi(r_{i})\le\updelta_{i}\,, i=1,\dotsc,\Bar{k}\bigr\}
\end{equation*}
is bounded and convex in some ball centered at $\Hat\updelta$ in $\RR^{\Bar{k}}$.
\end{proof}

\begin{definition}\label{D3.4}
For $\updelta\in\RR_{+}^{\Bar{k}}$ and
$\uplambda=(\uplambda_{1}\,\dotsc,\uplambda_{\Bar{k}})\transp
\in\RR^{\Bar{k}}_{+}$
define the running cost $g_{\updelta,\uplambda}$ by
\begin{equation*}
g_{\updelta,\uplambda}(x,u) \;\df\; r_{0}(x,u) + \sum_{i=1}^{\Bar{k}}
\uplambda_{i}\bigl(r_{i}(x,u)-\updelta_{i}\bigr)\,.
\end{equation*}
Also, for $\beta>0$ and $\Hat\updelta\in (0,\infty)^{\Bar{k}}$,
we define the set of Markov controls
\begin{equation*}
\sU_{\beta}(\updelta)\;\df\;\bigl\{v\in\Ussm\;\colon\, \uppi_{v}\in\sH(\updelta)\,,~
\uppi_{v}(r_{0})\le\beta\}\,,
\end{equation*}
and let $\sH_{\beta}(\updelta)$ denote the corresponding set of  ergodic occupation
measures.
\end{definition}

Lagrange multiplier theory provides us with the following.

\begin{lemma}\label{L3.5}
Suppose that $\updelta$ is feasible.
Then the following hold.
\begin{itemize}
\item[(i)]
There exists $\uplambda^{*}\in\RR^{\Bar{k}}_{+}$ such that
\begin{equation}\label{EE3.17}
\inf_{\uppi\,\in\,\sH(\updelta)}\;\uppi(r_{0}) \;=\;
\inf_{\uppi\,\in\,\eom}\;\uppi(g_{\updelta,\uplambda^{*}})\,.
\end{equation}
\item[(ii)]
Moreover, for any $\uppi^{*}\in\sH(\updelta)$
that attains the infimum of $\uppi\mapsto\uppi(r_{0})$
in $\sH(\updelta)$, we have
\begin{equation*}
\uppi^{*}(r_{0})\;=\;\uppi^{*}(g_{\updelta,\uplambda^{*}})\,,
\end{equation*}
and
\begin{equation*}
\uppi^{*}(g_{\updelta,\uplambda})\;\le\;\uppi^{*}(g_{\updelta,\uplambda^{*}})
\;\le\;
\uppi(g_{\updelta,\uplambda^{*}})
\qquad\forall\,(\uppi,\uplambda)\in\eom\times\RR^{\Bar{k}}_{+}\,.
\end{equation*}
\end{itemize}
\end{lemma}

\begin{proof}
The proof is standard.
See \cite[pp.~216--221]{Luenberger}.
\end{proof}

We next state the associated dynamic programming formulation of the ergodic
control problem under constraints.
Recall that $\tc_{\varepsilon}$
denotes the first hitting time of the ball $B_{\varepsilon}$,
for $\varepsilon>0$.

\begin{theorem} \label{T3.1}
Suppose that $\updelta\in (0,\infty)^{\Bar{k}}$ is feasible.
Let  $\uplambda^{*}\in\RR^{\Bar{k}}_{+}$ be as in Lemma~\ref{L3.5},
and $\uppi^{*}$ be any element of $\sH(\updelta)$ that attains the infimum
in \eqref{EE3.17}.
Then, the following hold.
\begin{itemize}
\item[(a)] 
There exists a $\varphi_{*}\in\Cc^{2}(\Rd)$
satisfying 
\begin{equation}\label{E-HJB-constr}
\min_{u\in\Act}\;\bigl[\Lg^{u}\varphi_{*}(x)
+ g_{\updelta,\uplambda^{*}}(x,u)\bigr]
\;=\; \uppi^{*}(g_{\updelta,\uplambda^{*}})\,,\quad
x\in\RR^{d}\,.
\end{equation}

\smallskip
\item[(b)]
With $\Lyap$ as in Hypothesis~\ref{HypA}, we have
$\varphi_{*}\in\order(\Lyap)$, and $\varphi^{-}_{*}\in\sorder(\Lyap)$.

\smallskip
\item[(c)] A stationary Markov control $v\in\Ussm$
is optimal if and only if it satisfies 
\begin{equation}\label{E-3.13}
\min_{u\in\Act}\;\sR_{\updelta,\uplambda^{*}}(x,\nabla\varphi_{*}(x);u)\;=\;
b\bigl(x, v(x)\bigr)\cdot \nabla\varphi_{*}(x)
+ g_{\updelta,\uplambda^{*}}\bigl(x,v(x)\bigr) \,,\quad
x\in\RR^{d}\,, 
\end{equation}
where
\begin{equation*}
\sR_{\updelta,\uplambda^{*}}(x,p;u) \;\df\; b\bigl(x,u\bigr)\cdot p
+ g_{\updelta,\uplambda^{*}}(x,u)\,.
\end{equation*}

\smallskip
\item[\textup{(}d\textup{)}]
The function $\varphi_{*}$ has the stochastic representation
\begin{align*}
\varphi_{*}(x)&\;=\;
\lim_{\varepsilon\searrow0}\;
\inf_{v\,\in\,\bigcup_{\beta>0}\,\sU_{\beta}(\updelta)}\;
\Exp^{v}_{x}\biggl[\int_{0}^{\tc_{\varepsilon}}
\Bigl(g_{\updelta,\uplambda^{*}}\bigl(X_{s},v(X_{s})\bigr)
-\uppi^{*}(g_{\updelta,\uplambda^{*}})\Bigr)\,\D{s}\biggr]\\[5pt]
&\;=\;
\lim_{\varepsilon\searrow0}\;\Exp^{\Bar{v}}_{x}\biggl[\int_{0}^{\tc_{\varepsilon}}
\Bigl(g_{\updelta,\uplambda^{*}}\bigl(X_{s},\Bar{v}(X_{s})\bigr)
-\uppi^{*}(g_{\updelta,\uplambda^{*}})\Bigr)\,\D{s}\biggr]\,,\nonumber
\end{align*}
for any $\Bar{v}\in\Usm$ that satisfies \eqref{E-3.13}.
\end{itemize}
\end{theorem}

\begin{proof}
Let $v^{*}\in\Ussm$
satisfy $\uppi^{*}(\D{x},\D{u})\df\mu_{v^{*}}(\D{x})\,v^{*}(\D{u}\mid x)$.
Since $\uppi^{*}(g_{\updelta,\uplambda^{*}})<\infty$,  there exists a function
$\varphi_{*}\in\Sobl^{2,p}(\RR^{d})$,
for any $p>d$, and such that $\varphi_{*}(0)=0$,
which solves the Poisson equation \cite[Lemma~3.7.8\,(ii)]{book}
\begin{equation}\label{E-Pois1}
\Lg^{v^{*}}\varphi_{*}(x) + g_{\updelta,\uplambda^{*}}\bigl(x,v^{*}(x)\bigr)
\;=\; \uppi^{*}(g_{\updelta,\uplambda^{*}})\,,\quad
x\in\RR^{d}\,,
\end{equation}
and satisfies, for all $\varepsilon>0$,
\begin{equation*}
\varphi_{*}(x) \;=\; \Exp^{v^{*}}_{x}\biggl[\int_{0}^{\tc_{\varepsilon}}
\Bigl(g_{\updelta,\uplambda^{*}}\bigl(X_{s},v^{*}(X_{s})\bigr)
-\uppi^{*}(g_{\updelta,\uplambda^{*}})\Bigr)\,
\D{s} + \varphi_{*}(X_{\tc_{\varepsilon}})\biggr]\qquad\forall x\in\RR^{d}\,.
\end{equation*}

Let $R>0$ be arbitrary, and select a Markov control $v_{R}$ satisfying
\begin{equation*}
v_{R}(x) \;=\;\begin{cases}
\Argmin_{u\in\Act}\, \sR_{\uplambda^{*}}(x,\nabla\varphi_{*}(x);u)
&\text{if~~}
\abs{x}<R\,,\\[2pt]
v^{*}(x)& \text{otherwise.}\end{cases}
\end{equation*}
It is clear that $v_{R}\in\Ussm$, and that if $\uppi_{R}$ denotes
the corresponding ergodic occupation measure, then we have $\uppi_{R}(r)<\infty$.
It follows by \eqref{E-Pois1} and the definition of $v_{R}$ that
\begin{equation}\label{E-Pois2}
\Lg^{v_{R}}\varphi_{*}(x) + g_{\updelta,\uplambda^{*}}\bigl(x,v_{R}(x)\bigr)
\;\le\; \uppi^{*}(g_{\updelta,\uplambda^{*}})\,,\quad
x\in\RR^{d}\,.
\end{equation}
By \eqref{E-Pois2} using \cite[Corollary~3.7.3]{book} we obtain
\begin{equation*}
\uppi_{R}\bigl(g_{\updelta,\uplambda^{*}}\bigr)
\;\le\;\uppi^{*}(g_{\updelta,\uplambda^{*}})\,.
\end{equation*}
However, since
$\uppi_{R}\bigl(g_{\updelta,\uplambda^{*}}\bigr)
\ge \uppi^{*}(g_{\updelta,\uplambda^{*}})$ by Lemma~\ref{L3.5},
it follows that we must have equality in \eqref{E-Pois2} a.e.\ in $\Rd$.
Therefore, since $R>0$ was arbitrary, we obtain \eqref{E-HJB-constr}.
By elliptic regularity, we have $\varphi_{*}\in\Cc^{2}(\Rd)$.
This proves part (a).

Continuing, note that by \eqref{EL3.3b} we have $\uppi^{*}(\Tilde{h})<\infty$,
and moreover
that $\sup_{\uppi\,\in\,\sH_{\beta}(\updelta)}\,\uppi(\Tilde{h})<\infty$
for all $\beta>0$.
Thus we can follow the approach in Section~3.5 of \cite{ABP14},
by considering the perturbed problem with running cost of the form
$g_{\updelta,\uplambda^{*}}+\varepsilon\Tilde{h}$ and then take
limits as $\varepsilon\searrow0$.
Parts (b)--(d) then follow as in Theorem~3.4 and Lemma~3.10 of \cite{ABP14}.
\end{proof}

Concerning uniqueness, the analogue of Theorem~3.5 in \cite{ABP14} holds, which we
quote next.
The proof follows that of \cite[Theorem~3.5]{ABP14} and is therefore omitted.

\begin{theorem}\label{T-unique}
Let the hypotheses of Theorem~\ref{T3.1} hold, and
$(\Hat{\varphi},\Hat\varrho)\in\Cc^{2}(\Rd)\times\RR$ be a solution of
\begin{equation}\label{E-HJB-hat2}
\min_{u\in\Act}\;\bigl[\Lg^{u}\Hat{\varphi}(x)+g_{\updelta,\uplambda^{*}}(x,u)\bigr]
\;=\;\Hat\varrho\,,
\end{equation}
such that $\Hat{\varphi}^{-}\in\sorder(\Lyap)$ and $\Hat{\varphi}(0)=0$.
Then the following hold:
\begin{itemize}
\item[\textup{(}a\textup{)}]
Any measurable selector $\Hat{v}$ from the minimizer of \eqref{E-HJB-hat2}
is in $\Ussm$ and $\uppi_{\Hat{v}}(g_{\updelta,\uplambda^{*}}) < \infty$.
\smallskip
\item[\textup{(}b\textup{)}]
If either $\Hat\varrho\le \uppi^{*}(g_{\updelta,\uplambda^{*}})$,
or  $\Hat{\varphi}\in\order\bigl(\min_{u\in\Act}\;\Tilde{h}(\cdot\,,u)\bigr)$,
then necessarily
$\Hat\varrho= \uppi^{*}(g_{\updelta,\uplambda^{*}})$,
and $\Hat{\varphi}=\varphi_{*}$.
\end{itemize}
\end{theorem}

We finish this section by presenting an analogues to
\cite[Theorems~4.1 and 4.2]{ABP14}
(see also Lemma~\ref{L3.2})
for the ergodic control problem under constraints.
Let $v_{0}\in\Ussm$ be any control such that $\uppi_{v_{0}}(r)<\infty$.
For $j=0,1,\dotsc,\Bar{k}$, define the truncated running costs $r^{R}_{j}$
relative to $r_{j}$ as in \eqref{rR}.
We consider the ergodic control problem under constraints
in \eqref{E-ec-c1}--\eqref{E-ec-c2}
for the family of controlled diffusions, parameterized by $R\in\NN$,
given by \eqref{E-sdeR}
with running costs $r_{i}\equiv r^{R}_{j}(x,u)$, $j=0,1,\dotsc,\Bar{k}$.
Recall that, as defined in Section~\ref{sec-structural},
$\eom(R)$ denotes the set of ergodic occupation
measures of \eqref{E-sdeR}.
We also let $\sH(\updelta;R)$, $\sH^{\mathrm{o}}(\updelta;R)$ be defined
as in \eqref{E-sH} relative to the set $\eom(R)$.
We have the following theorem.

\begin{theorem}\label{T3.3}
Suppose that $\Hat\updelta\in (0,\infty)^{\Bar{k}}$ is feasible,
and that $\varphi_{0}$ defined in \eqref{E-vphi0} satisfies
$\varphi_{0}\in\order\bigl(\min_{u\in\Act}\;\Tilde{h}(\cdot\,,u)\bigr)$.
Then the following are true.
\begin{itemize}
\item[(a)]
There exists $R_{0}>0$
such that
\begin{equation*}
\sH^{\mathrm{o}}(\Hat\updelta;R)
\cap\{\uppi\in\eom(R)\;\colon\, \uppi(r_{0})<\infty\}
\;\ne\;\varnothing\qquad\forall\,R\ge R_{0}\,.
\end{equation*}
\item[(b)]
It holds that
\begin{equation*}
\inf_{\uppi\,\in\,\sH(\Hat\updelta;R)}\;\uppi(r_{0})
\;\xrightarrow[R\to\infty]{}\;
\inf_{\uppi\,\in\,\sH(\Hat\updelta)}\;\uppi(r_{0}) \,.
\end{equation*}
\end{itemize}
\end{theorem}

\begin{proof}
Let $\varepsilon>0$ be given.
By Lemma~\ref{L3.4}, for all sufficiently small $\varepsilon>0$, there exist
$\updelta_{i}^{\varepsilon}<\Hat\updelta_{i}$,
$i=1,\dotsc,\Bar{k}$, such that $\updelta^{\varepsilon}$ is feasible and
\begin{equation}\label{ET3.3a}
\inf_{\uppi\,\in\,\sH(\updelta^{\varepsilon})}\;\uppi(r_{0})
\;\le\; \inf_{\uppi\,\in\,\sH(\Hat\updelta)}\;\uppi(r_{0})
+ \tfrac{\varepsilon}{4}\,.
\end{equation}
For $\Tilde\varepsilon>0$, let
$r_{\Tilde\varepsilon}\df r_{0}+\Tilde\varepsilon\,\Tilde{h}$.
By \eqref{E-key} we have 
\begin{equation*}
\uppi(r_{0})\;\le\; \uppi(r_{\Tilde\varepsilon})\;\le\;
(1+k_{0}\Tilde\varepsilon)\,\uppi(r_{0})+
k_{0}\Tilde\varepsilon+ k_{0}\Tilde\varepsilon\,\sum_{i=1}^{\Bar{k}}\updelta_{i}
\qquad\forall\, \uppi\in\sH(\updelta)\,.
\end{equation*}
Therefore, for any $\varepsilon>0$,
we can choose $\Tilde\varepsilon>0$
small enough so that
\begin{equation}\label{ET3.3b}
\inf_{\uppi\,\in\,\sH(\updelta^{\varepsilon})}\;\uppi(r_{\Tilde\varepsilon})
\;\le\;\inf_{\uppi\,\in\,\sH(\updelta^{\varepsilon})}\;\uppi(r_{0})
+ \tfrac{\varepsilon}{4}\,.
\end{equation}
Let
\begin{equation*}
g_{\Tilde\varepsilon,\updelta^\varepsilon,\uplambda}(x,u) \;\df\;
r_{\Tilde\varepsilon}(x,u) + \sum_{i=1}^{\Bar{k}}
\uplambda_{i}\bigl(r_{i}(x,u)-\updelta_{i}^\varepsilon\bigr)\,.
\end{equation*}
By Lemmas~\ref{L3.3} and \ref{L3.5}
there exist $\uplambda^{*}\in\RR^{\Bar{k}}_{+}$ and
$\uppi^{*}\in\sH(\updelta^{\varepsilon})$
such that
\begin{equation}\label{ET3.3c}
\uppi^{*}(r_{\Tilde\varepsilon})\;=\;
\inf_{\uppi\,\in\,\sH(\updelta^{\varepsilon})}\;\uppi(r_{\Tilde\varepsilon}) \;=\;
\inf_{\uppi\,\in\,\sH(\updelta^{\varepsilon})}\;
\uppi(g_{\Tilde\varepsilon,\updelta^\varepsilon,\uplambda^{*}})
\;=\;\uppi^{*}(g_{\Tilde\varepsilon,\updelta^\varepsilon,\uplambda^{*}})\,.
\end{equation}

Define the truncated running cost
$g_{\Tilde\varepsilon,\updelta^\varepsilon,\uplambda^{*}}^{R}$
relative to $g_{\Tilde\varepsilon,\updelta^\varepsilon,\uplambda^{*}}$
as in \eqref{rR}.
Since $\uppi_{v_{0}}(g_{\Tilde\varepsilon,\updelta^\varepsilon,\uplambda^{*}}^{R})$
is finite, the hypotheses of Lemma~\ref{L3.2} are satisfied.
Let $\Bar\eom(R)$ denote the collection of ergodic occupation measures
in $\eom(R)$ which are optimal for
for $g_{\Tilde\varepsilon,\updelta^\varepsilon,\uplambda^{*}}^{R}$.
Therefore, it follows by Lemma~\ref{L3.2} that
$\{\Bar\eom(R)\;\colon\,R>0\}$
is tight, and any limit point of $\Bar\uppi^{R}\in\Bar\eom(R)$ as $R\to\infty$
satisfies \eqref{ET3.3c}.
Since $r_{i}\le\Tilde{h}$ it follows by dominated convergence
that 
\begin{equation*}
\limsup_{R\to\infty}\;\Bar\uppi^{R}(r^{R}_{i})\;\le\;\updelta^{\varepsilon}_i\;<\;
\Hat\updelta_i\,, \quad i \;=\;1,\dotsc,\Bar{k} \,.
\end{equation*} 
which establishes part (a).

Therefore, there exists $R_{0}>0$ such that $\Bar\uppi^{R}\in\sH(\Hat\updelta,R)$
for all $R>R_{0}$, and by \eqref{ET3.3c},
\begin{equation}\label{ET3.3d}
\Bar\uppi^{R}(r_{\Tilde\varepsilon})\;\le\;
\inf_{\uppi\,\in\,\sH(\updelta^{\varepsilon})}\;\uppi(r_{\Tilde\varepsilon})
+\tfrac{\varepsilon}{2}\qquad \forall\,R>R_{0}\,.
\end{equation}
Combining \eqref{ET3.3a}--\eqref{ET3.3b} and \eqref{ET3.3d} we obtain
\begin{equation*}
\Bar\uppi^{R}(r_{0})\;\le\;
\Bar\uppi^{R}(r_{\Tilde\varepsilon})\;\le\;
\inf_{\uppi\,\in\,\sH(\Hat\updelta)}\;\uppi(r_{0}) +\varepsilon\,,
\end{equation*}
which establishes part (b).
The proof is complete.
\end{proof}

Let $\updelta\in (0,\infty)^{\Bar{k}}$ and $R>0$.
Provided $\sH^{\mathrm{o}}(\updelta;R)\ne\varnothing$ we denote
by $\uplambda^{*}_{R}=(\uplambda_{1,R}\,\dotsc,\uplambda_{\Bar{k},R})\transp
\in\RR^{\Bar{k}}_{+}$ any such vector satisfying
\begin{equation*}
\inf_{\uppi\,\in\,\sH(\updelta,R)}\;\uppi(r_{0}) \;=\;
\inf_{\uppi\,\in\,\eom(R)}\;\uppi(g_{\updelta,\uplambda^{*}_{R}})\,,
\end{equation*}
and by $\uppi^{*}_{R}$, any member of $\sH(\updelta,R)$ that attains this
infimum.
It follows by Theorem~\ref{T3.3}~(a) that, provided
$\sH^{\mathrm{o}}(\updelta)\ne\varnothing$, then
$\sH^{\mathrm{o}}(\updelta;R)\ne\varnothing$ for all
$R$ sufficiently large.
Clearly, $\uppi^{*}_{R}$ satisfies \eqref{EL3.3b} and
$R\mapsto\uppi^{*}_{R}(r_{0})$ is nonincreasing.
Therefore $\{\uppi^{*}_{R}\}$ is a tight family.
It then follows by Theorem~\ref{T3.3}~(b) that
any limit point
of $\uppi^{*}_{R}$ as $R\to\infty$ attains the minimum
of $\uppi\to\uppi(r_{0})$ in $\sH(\updelta)$.
Concerning the convergence of the solutions to the associated
HJB equations we have the following.

\begin{theorem}\label{T3.4}
Suppose that $\updelta\in (0,\infty)^{\Bar{k}}$ is feasible.
Let $\Lg_{R}^{u}$ denote the controlled extended generator corresponding
to the diffusion in \eqref{E-sdeR},
$\varphi_{0}$ be as in \eqref{E-vphi0},
and $\uplambda^{*}_{R}$,
$g^{R}_{\updelta,\uplambda^{*}}$ defined as in \eqref{rR} relative
to the running cost $g_{\updelta,\uplambda^{*}}$,
$\uppi^{*}_{R}$ be as defined in the previous paragraph.
Then there exists $R_{0}>0$ such that for all $R>R_{0}$
the HJB equation
\begin{equation}\label{E-HJB-n}
\min_{u\in\Act}\;\bigl[\Lg_{R}^{u}V_{R}(x)
+g^{R}_{\updelta,\uplambda^{*}_{R}}(x,u)\bigr]
\;=\;\uppi^{*}_{R}(r_{0})\,,
\end{equation}
has a solution $V_{R}$ in $\Sobl^{2,p}(\RR^{d})$,
for any $p>d$, with $V_{R}(0)=0$,
and such that the restriction of $V_{R}$ on $B_{R}$ is in $\Cc^{2}(B_{R})$.
Also, the following hold:
\begin{itemize}
\item[\textup{(}i\textup{)}]
there exists a constant $C_{0}$, independent
of $R$, such
that $V_{R} \le C_{0}+2\varphi_{0}$
for all $R>R_{0}$;
\smallskip
\item[\textup{(}ii\textup{)}]
$(V_{R})^{-}\in\sorder(\Lyap+\varphi_{0})$ uniformly over $R>R_{0}$;
\smallskip
\item[\textup{(}iii\textup{)}]
Every $\uppi^{*}_{R}$ corresponds to
a stationary Markov control $v\in\Ussm$
that satisfies 
\begin{equation}\label{E-3.13new}
\min_{u\in\Act}\;
\bigl[b^{R}(x,u)\cdot \nabla V_{R}(x)
+ g^{R}_{\updelta,\uplambda^{*}_{R}}(x,u)\bigr]\;=\;
b\bigl(x,v(x)\bigr)\cdot \nabla V_{R}(x)
+ g_{\updelta,\uplambda^{*}_{R}}(x,v(x)) \,,\quad
\text{a.e.}~x\in\RR^{d}\,.
\end{equation}
\end{itemize}
Let $\varphi_{*}$ and $\uplambda^{*}$ be as in Theorem~\ref{T3.1}.
Then, under the additional hypothesis that
\begin{equation*}
\varphi_{0}\in\order\bigl(\min_{u\in\Act}\;\Tilde{h}(\cdot\,,u)\bigr)\,,
\end{equation*}
for every sequence $R\nearrow\infty$ there exists a subsequence
along which it holds that $V_{R}\to\varphi_{*}$ and
$\uplambda^{*}_{R}\to\uplambda^{*}$.
Also, if $\Hat{v}_{R}$ is a measurable selector from the minimizer
of \eqref{E-3.13new} then any limit point of $\Hat{v}_{R}$
in the topology of Markov controls as $R\to\infty$ is a measurable
selector from the minimizer of \eqref{E-3.13}.
\end{theorem}

\begin{proof}
We can start from the perturbed problem with running cost of the form
$g_{\updelta,\uplambda^{*}_{R}}+\varepsilon\Tilde{h}$ to
establish \eqref{E-3.13new} in Section~3.5 of \cite{ABP14},
and then take limits as $\varepsilon\searrow0$.
Parts (i) and (ii) can be established by following the proof
of \cite[Theorem~4.1]{ABP14}.
Convergence to \eqref{E-3.13} as $R\to\infty$ follows as in the proof
of \cite[Theorem~4.2]{ABP14}.
\end{proof}

\medskip
\section{Recurrence Properties of the Controlled Diffusions}   \label{sec-stabilizability}

In this section, we show that the limiting diffusions for a multiclass
multi-pool network satisfy Hypothesis~\ref{HypA} relative
to the running cost in \eqref{E-cost} for any
value of the parameters.
Also, provided $\gamma\ne0$, Hypothesis~\ref{HypB} is also satisfied. 
The proofs rely on a recursive
leaf elimination algorithm which we introduce next. 

\subsection{A leaf elimination algorithm and drift representation}
\label{sec-leaf-alg}
We now present a leaf elimination algorithm and prove some properties. 
Recall the linear map $G$ defined in \eqref{diff-map} and the associated matrix
$\Psi$ in \eqref{Psi}, and also the map $\widehat{G}$ defined in \eqref{diff-map2}. 

\begin{definition}
Let $\cG\bigl(\cI\cup\cJ,\cE,(\alpha,\beta)\bigr)$ denote the
labeled graph, whose nodes are labeled by $(\alpha,\beta)$,
i.e., each node $i\in\cI$ has the label $\alpha_{i}$, and 
each node $j\in\cJ$ has the label $\beta_{j}$.
The graph $\cG$ is a tree and there is a one to one correspondence
between this graph and the matrix $\Psi=\Psi(\alpha,\beta)$
defined in \eqref{Psi}.
We denote this correspondence by $\Psi\sim \cG$.

Let $\Psi^{(-i)}$ denote the $(I-1)\times J$ submatrix of $\Psi$
obtained after eliminating the $i^\text{th}$ row of $\Psi$.
Similarly, $\Psi_{(-j)}$ is the $I\times (J-1)$ submatrix resulting
after the elimination of the $j^\text{th}$ column.

If $\him\in\cI$ is a leaf of
$\cG\bigl(\cI\cup\cJ,\cE,(\alpha,\beta)\bigr)$,  we let $j_{\him}\in\cJ$
denote the unique node such that $(\him,j_{\him})\in\cE$
and define
\begin{equation*}
(\alpha,\beta)^{(-\him)}\;\df\; \bigl(\alpha_{1},\dotsc,\alpha_{\him-1},
\alpha_{\him+1},\dotsc,\alpha_{I},\beta_{1},\dotsc,\beta_{j_{\him}-1},
\beta_{j_{\him}}-\alpha_{\him},\beta_{j_{\him}+1},
\dotsc,\beta_{J}\bigr)\,, 
\end{equation*}
i.e., $(\alpha,\beta)^{(-\him)}\in\RR^{I-1+J}$ is the vector of
parameters obtained after removing $\alpha_{\him}$ and replacing
$\beta_{j_{\him}}$ with $\beta_{j_{\him}}-\alpha_{\him}$.
Similarly, if $\hjm\in\cJ$ is a leaf, we define $i_{\hjm}$
and $(\alpha,\beta)_{(-\hjm)}$ in a completely analogous manner. 
\hfill $\Box$
\end{definition}

\begin{lemma}
If $\him\in\cI$ and/or
$\hjm\in\cJ$ are leafs of $\cG\bigl(\cI\cup\cJ,\cE,(\alpha,\beta)\bigr)$,
then
\begin{align*}
\Psi^{(-\him)}(\alpha,\beta)&\;\sim\;
\cG\Bigl((\cI\setminus\{\him\})\cup\cJ,
\cE\setminus\{(\him,j_{\him})\},
(\alpha,\beta)^{(-\him)}\Bigr) \,,\\
\Psi_{(-\hjm)}(\alpha,\beta)&\;\sim\;
\cG\Bigl(\cI\cup(\cJ\setminus\{\hjm\}),
\cE\setminus\{(i_{\hjm},\hjm)\},
(\alpha,\beta)_{(-\hjm)}\Bigr)\,.
\end{align*}
\end{lemma}

\begin{proof}
If $\him\in\cI$ is a leaf of
$\cG\bigl(\cI\cup\cJ,\cE,(\alpha,\beta)\bigr)$,
 then
$\psi_{\him,j_{\him}}$ is the unique non-zero
element in the $\him^\text{th}$ row of $\Psi(\alpha,\beta)$.
Therefore, the equivalence follows by the fact that the concatenation
of
$\Psi^{(-\him)}(\alpha,\beta)$ and row $\him$
of $\Psi(\alpha,\beta)$ has the same row and column sums
as $\Psi(\alpha,\beta)$.
Similarly if $\hjm\in\cJ$ is a leaf.
\end{proof}

\begin{definition}
In the interest of simplifying the notation, for a labeled
tree $\cG=\cG\bigl(\cI\cup\cJ,\cE,(\alpha,\beta)\bigr)$
we denote
\begin{equation*}
\cG^{(-\him)}\;\df\;
\cG\Bigl((\cI\setminus\{\him\})\cup\cJ,
\cE\setminus\{(\him,j_{\him})\},
(\alpha,\beta)^{(-\him)}\Bigr)\,,
\end{equation*}
and
\begin{equation*}
\cG_{(-\hjm)}\;\df\;
\cG\Bigl(\cI\cup(\cJ\setminus\{\hjm\}),
\cE\setminus\{(i_{\hjm},\hjm)\},
(\alpha,\beta)_{(-\hjm)}\Bigr)\,,
\end{equation*}
for leaves $\him\in\cI$ and $\hjm\in\cJ$, respectively.
\end{definition}

We now present a leaf elimination algorithm, which starts from a server leaf
elimination. A similar algorithm can start from a customer leaf elimination.

\begin{algo}
Consider the tree $\cG=\cG\bigl(\cI\cup\cJ,\cE,(\alpha,\beta)\bigr)$
as described above.

\noindent\emph{Server Leaf Elimination.}
Let $\cJ_{\text{leaf}}\subset\cJ$ be the collection of all leaves of $\cG$
which are members of $\cJ$.
We eliminate each $\hjm\in\cJ_{\text{leaf}}$ sequentially in any order,
each time replacing $\cG$ by $\cG_{(-\hjm)}$
and setting $\psi_{i_{\hjm}\hjm}=\beta_{\hjm}$.
Let $\cG^{1}=\cG(\cI^{1}\cup\cJ^{1},\cE^{1},(\alpha^{1},\beta^{1}))$
denote the graph obtained.
Note that $\cI^{1}=\cI$
and $\cJ^{1}=\cJ\setminus\cJ_{\text{leaf}}$, and all the leaves
of $\cG^{1}$ are in $\cI$.
Note also that since $\cG^{1}$ is a tree, it contains at least two leaves
unless its maximum degree equals $1$.
Let $\widetilde\varPsi^{1}$ denote the collection of nonzero
elements of $\Psi$ thus far defined.

Given $G^{k}=\cG(\cI^{k}\cup\cJ^{k},\cE^{k},(\alpha^{k},\beta^{k}))$,
for each $k=1,2,3,\dotsc, I-1$, we perform the following:

\begin{enumerate}
\item[\textup{(i)}]
Choose any leaf $\him\in\cI^{k}$ and
set $\psi_{\him j_{\him}}=\alpha^{k}_{\him}$
and $\pi(\him)=k$.
Replace $\cG^{k}$ with $\bigl(\cG^{k}\bigr)^{(-\him)}$.
Let $\widetilde\varPsi^{k+1}=\widetilde\varPsi^{k}\cup\{\psi_{\him j_{\him}}\}$.
\item[\textup{(ii)}]
For $\bigl(\cG^{k}\bigr)^{(-\him)}$ obtained in (i),  perform
the server leaf elimination as described above, and
denote the resulting graph by $\cG^{k+1}$, and by
$\widetilde\varPsi^{k+1}$ denote the collection of nonzero
elements of $\Psi$ thus far defined.
\end{enumerate}

At step $I-1$,  the resulting graph $\cG^I$ has a maximum degree of zero,
where $\cI^{k}=\{\him\}$ is a singleton
and $\cJ^{k}$ is empty and $\Psi$ contains exactly $I+J-1$ non-zero elements.
We set $\pi(\him)=I$. 
\end{algo}

\begin{remark}
We remark that in the first step of server leaf elimination, all leaves in
$\cJ$ are removed while in each customer leaf elimination, only one leaf
in $\cI$ (if more than one) is removed. Thus, exactly $I$ steps
of customer leaf elimination are conducted in the algorithm. 
The input of the algorithm is a tree $\cG$ with the vertices $\cI\cup \cJ$,
the edges $\cE$ and the indices $(\alpha, \beta)$.
The output of the algorithm is the matrix $\Psi=\Psi(\alpha, \beta)$---the unique
solution to the linear map $G$ defined in \eqref{diff-map},
and the permutation of the leaves $\cI$ which tracks the order of the leaves
being eliminated, that is, for each $k =1,2,\dotsc,I$, $\pi(i)=k$
for some $i \in \cI$. Note that the permutation $\pi$ may not be unique,
but the matrix $\Psi$ is unique for a given tree $\cG$.
The elements of the matrix $\Psi$ determine the drift $b(x,u)= b(x,(u^c,u^s))$
by \eqref{diff-drift}.
It is shown in the lemma below that the nonzero elements of $\Psi$ are
linear functions of $(\alpha,\beta)$, which provides an important insight
on the structure of the drift $b(x,u)$; see Lemma~\ref{lem-drift}. 
\end{remark}

\begin{lemma}\label{L-perm}
Let $\pi$ denote the permutation of $\cI$ defined in the leaf elimination
algorithm, and $\pi^{-1}$ denote its inverse. For each  $k\in\cI$,
\begin{itemize}
\item[(a)]
the elements of the matrix
$\widetilde\varPsi^{k}$ are functions
of
\begin{equation*}
\{\alpha_{\pi^{-1}(1)}\,,\dotsc,\alpha_{\pi^{-1}(k-1)},\beta\}\,;\end{equation*}
\item[(b)]
the set
\begin{equation*}
\bigl\{\psi_{ij}\in\widetilde\varPsi^{k}\,\colon
i=\pi^{-1}(1),\dotsc,\pi^{-1}(k),\;j\in\cJ\bigr\}
\end{equation*}
and the set of nonzero elements of rows
$\pi^{-1}(1),\dotsc,\pi^{-1}(k)$ of $\Psi$ are equal;
\item[(c)]  there exists a linear function $F_k$ such that
\begin{equation*}
\alpha_{\pi^{-1}(k)}^{k} \;=\;\alpha_{ \pi^{-1}(k)  }
- F_k(\alpha_{\pi^{-1}(1)}\,, \dotsc,\alpha_{\pi^{-1}(k-1)},\beta)\,.
\end{equation*}
\end{itemize}
\end{lemma}

\begin{proof}
This is evident from the incremental definition of $\Psi$ in the algorithm.
\end{proof}

\begin{lemma} \label{lem-drift}
The drift $b(x, u)= b(x, (u^c, u^s))$ in the limiting diffusion $X$
in \eqref{hatX}
can be expressed as
\begin{equation} \label{diff-drift2}
b(x, u) \;=\;- B_1 (x -  (e\cdot x)^{+} u^c) + (e\cdot x)^{-} B_2 u^s
- (e \cdot x)^{+} \varGamma u^c  + \ell \,,
\end{equation}
where $B_1$ is a lower-diagonal $I\times I$ matrix with positive diagonal
elements, $B_2$ is an $I \times J$ matrix and
$\varGamma= \diag\{\gamma_1,\dotsc, \gamma_I\}$.
\end{lemma}

\begin{proof}
We perform the leaf elimination algorithm and reorder the indices in $\cI$
according to the permutation $\pi$. Thus, leaf $i \in \cI$ is eliminated in step
$i$ of the customer leaf elimination. 
Let $j_i\in\cJ$ denote the unique node corresponding to $i\in\cI$,
when $i$ is eliminated as a leaf in step $i$  of the algorithm.
It is important to note that, with respect to the reordered indices, 
the matrix
$\widehat{G}^{0}(x)$ (see Remark~\ref{R-hG}) takes the following form
\begin{equation*}
\widehat{G}_{i,j}^{0}(x)\;=\;
\begin{cases}
x_{i} + \widetilde{G}_{i{j_i}}(x_{1},\dotsc,x_{i-1})&\text{for}~j=j_i\,,\\[3pt]
\widetilde{G}_{ij}(x_{1},\dotsc,x_{i-1})&\text{for}~i\sim j\,,
\;j\ne j_i \,, \\[3pt]
0&\text{otherwise,}
\end{cases}
\end{equation*}
where each $\widetilde{G}_{ij}$ is a \emph{linear} function of its arguments.
As a result, by Lemma~\ref{L-perm}, the drift takes the form
\begin{equation}\label{E-drift2}
b_{i}\bigl(x,u\bigr)\;=\; - \mu_{ij_i} x_{i}
+ \Tilde{b}_{i}(x_{1},\dotsc,x_{i-1}) + \Tilde{F}_{i}\bigl(\pex u^c,\nex u^s \bigr)
- \gamma_i\, (e\cdot x)^{+} u_i^c + \ell_i\,.
\end{equation}
Two things are important to note: (a) $\Tilde{F}_{i}$ is a linear function,
and (b) $\mu_{ij_i}>0$ (since $i\sim j_i$).

Let $\widehat{b}$ denote the vector field
\begin{equation}\label{E-hatb}
\widehat{b}_{i}(x) \;\df\;
- \mu_{ij_i} x_{i}
+ \Tilde{b}_{i}(x_{1},\dotsc,x_{i-1})\,.
\end{equation}
Then $\widehat{b}$ is a linear vector field corresponding to a lower-diagonal
matrix with negative diagonal elements, and this is denoted by $- B_1$.
The form of the drift in \eqref{diff-drift2} then readily follows by the
leaf elimination algorithm and \eqref{diff-drift}.  
\end{proof}

\begin{remark}
By the representation of the drift $b(x, u)$ in \eqref{diff-drift2},
the limiting diffusion $X$ can be classified as a piecewise-linear
controlled diffusion as discussed in Section~3.3 of \cite{ABP14}.
The difference of the drift $b(x,u)$ from that in \cite{ABP14} lies in two aspects:
 (i) there is an additional term $(e\cdot x)^{-} B_2 u^s$, and
 (ii)  $B_1$ may not be an $M$-Matrix (see, e.g., the $B_1$ matrices in the
 $W$ model and the model in Example~\ref{Ex4} below). 
\end{remark}

\smallskip

\subsection{Examples} \label{sec-example}
In this section, we provide several examples to illustrate the leaf
elimination algorithm, including the classical ``N", ``M", ``W" models and
the non-standard models that cannot be solved in \cite{Atar-05a, Atar-05b}.
Note that in Assumption 3 of \cite{Atar-05b} (and in Theorem~1 of \cite{Atar-05a}),
it is required that either of the following conditions holds: (i) the service
rates $\mu_{ij}$ are either class or pool dependent, and $\gamma_i =0$
for all $i \in \cI$;  (ii) the tree $\cG$ is of diameter 3 at most and in addition,
$\gamma_i \le \mu_{ij}$ for each $i\sim j $ in $\cG$.
We do not impose any of these conditions
in asserting Hypotheses~\ref{HypA} and \ref{HypB} later
in Section~\ref{sec-verify}.

\begin{example}[The ``N" model]
Let $\cI = \{1,2\}$, $\cJ =\{1, 2\}$ and $\cE = \{1\sim 1, 1\sim 2, 2\sim 2\}$.
The matrix $\Psi$ takes the form
$\Psi(\alpha,\beta) =\begin{bmatrix} \beta_1 & \alpha_1 -\beta_1 \\
0 & \alpha_2 \end{bmatrix} $
and the permutation $\pi$ satisfies $\pi^{-1}(k) =k $ for $k=1,2$. 
The matrices $B_1$ and $B_2$ in the drift $b(x, u)$ are
$B_1 = \diag\{\mu_{12},  \mu_{22}\}$ and $B_2 = \diag\{\mu_{11} - \mu_{12}, 0\}$. 
\end{example}

\begin{remark}\label{R-hG-0}
Recall $\widehat{G}^{0}(x)$ in Remark~\ref{R-hG}. 
Applying the leaf elimination algorithm, there may be more than one
realizations of $\widehat{G}^{0}(x)$.
For example in the `N' network, the solution can be expressed as
$\Psi(\alpha,\beta) =\begin{bmatrix} \beta_1 & \alpha_1 -\beta_1 \\
0 & \alpha_2 \end{bmatrix}$,
or $\Psi(\alpha,\beta) =\begin{bmatrix} \beta_1 & \beta_2 -\alpha_2 \\
0 & \alpha_2 \end{bmatrix}$, and these give different answers when
$u\equiv0$.
It depends on the permutation order in the implementation of elimination, i.e.,
which pair of nodes is eliminated last. 
\end{remark}

\begin{example}[The ``W" model]\label{Ex2}
Let $\cI = \{1, 2, 3\}$, $\cJ = \{1, 2\}$ and
$\cE = \{ 1\sim 1, 2\sim 1, 2 \sim 2, 3 \sim 2\}$. 
Following the algorithm, we obtain that the matrix $\Psi$ takes the form
\begin{equation*}
\Psi(\alpha,\beta)\;=\;\begin{bmatrix} \alpha_1 &  0 \\
\beta_1 - \alpha_1 & \alpha_2 - (\beta_1 - \alpha_1) \\ 0 & \alpha_3
\end{bmatrix}\,,
\end{equation*}
and the permutation $\pi$ satisfies $\pi^{-1}(k) =k $ for $k=1,2, 3$.
The matrices $B_1$ and $B_2$ in the drift $b(x, u)$ are
\begin{equation*}
B_1 \;=\; \begin{bmatrix} \mu_{11} & 0 &0 \\
\mu_{21} + \mu_{22} &  \mu_{22} & 0 \\ 0 & 0 & \mu_{32} \end{bmatrix}
\qquad\text{and}\qquad 
B_2 \;=\; \begin{bmatrix}0 & 0  \\ \mu_{21}-\mu_{22}& 0  \\
0 & 0  \end{bmatrix}\,.
\end{equation*}
\end{example}

\begin{example}[The ``M" model]
Let $\cI = \{1, 2\}$, $\cJ= \{1, 2, 3\}$, and
$\cE = \{1\sim 1, 1\sim 2, 2\sim2, 2\sim 3 \}$. 
The matrix $\Psi$ takes the form
\begin{equation*}
\Psi(\alpha,\beta) =
\begin{bmatrix} \beta_1 & \alpha_1 - \beta_1 &  0 \\[2pt]
0 & \alpha_2 - \beta_3 & \beta_3  \end{bmatrix}\,,
\end{equation*}
and the permutation $\pi$ satisfies $\pi^{-1}(k) =k $ for $k=1,2$. 
The matrices $B_1$ and $B_2$ in the drift $b(x, u)$ are
\begin{equation*}
 B_1\;=\;\diag\{\mu_{12}, \mu_{22}\}
\qquad\text{and}\qquad 
B_2\;=\;\begin{bmatrix}\mu_{11}-\mu_{12} & 0 & 0 \\[2pt]
0& 0 & \mu_{23}- \mu_{22}  \end{bmatrix}\,.
\end{equation*}
\end{example}

\begin{example}\label{Ex4}
Let $\cI = \{1, 2, 3, 4\}$, $\cJ= \{1, 2, 3\}$ 
and $\cE = \{1\sim 1, 2\sim 1, 2\sim2, 2\sim 3, 3\sim 3, 4\sim 3 \}$. 
We obtain
\begin{equation*}
\Psi(\alpha,\beta) \;=\;
\begin{bmatrix} \alpha_1 & 0 &  0 \\
\beta_1 - \alpha_1 & \beta_2 & (\alpha_2 - \beta_2) - (\beta_1 - \alpha_1) \\
0 & 0 & \alpha_3 \\ 0 & 0 & \alpha_4 \end{bmatrix}\,,
\end{equation*}
and the permutation $\pi$ satisfies $\pi^{-1}(k) =k $ for $k=1,2,3,4$. 
The matrices $B_1$ and $B_2$ in the drift $b(x, u)$ are
\begin{equation*}
B_1 \;=\; \begin{bmatrix} \mu_{11} & 0 &0 & 0 \\
-\mu_{21} + \mu_{23} & \mu_{23} & 0 & 0 \\
0 & 0 & \mu_{33} & 0 \\ 0 & 0 & 0 & \mu_{43} \end{bmatrix}
\qquad\text{and}\qquad
B_2 \;=\;\begin{bmatrix}0 & 0 & 0 \\
-\mu_{21} - \mu_{23} &  - \mu_{23} & 0 \\
0 & 0 &0  \\ 0 & 0 & 0  \end{bmatrix}\,.
\end{equation*}
\end{example}

\smallskip

\subsection{Verification of Hypotheses~\ref{HypA} and \ref{HypB}}
\label{sec-verify}

In this section we show that the controlled diffusions $X$ in \eqref{hatX} for
the multiclass multi-pool networks
satisfy Hypotheses~\ref{HypA} and \ref{HypB}. 

\begin{theorem}\label{T-PA} 
For the unconstrained ergodic control problem \eqref{diff-opt}
under a running cost $r$ in \eqref{E-cost} with
strictly positive vectors $\xi$ and $\zeta$, Hypothesis~\ref{HypA} holds for
$\cK=\cK_{\delta}$ defined by
\begin{equation}\label{E-cK}
\cK_{\delta}\;\df\; \bigl\{ x\in\RI\,\colon \abs{e\cdot x} > \delta \abs{x}\bigr\}
\end{equation}
for some $\delta>0$ small enough, and for a function
$h(x) \df \Tilde{C} |x|^m$ with some positive $\Tilde{C}$.
\end{theorem}

\begin{proof}
Recall the form of the drift $b(x, u)$ in \eqref{diff-drift2} in
Lemma~\ref{lem-drift}.
The set $\cK_{\delta}$ in \eqref{E-cK} is an open convex cone, and the running cost
function $r(x,u) = r(x, (u^c, u^s))$ in \eqref{E-cost} is inf-compact on
$\cK_{\delta}$.
Define $\Lyap\in\Cc^2(\RR^I)$ by
$\Lyap(x)\df(x\transp Q x)^{\nicefrac{m}{2}}$ for $|x|\ge 1$,
where $m$ is as given in \eqref{E-cost},
and the matrix $Q$ is a diagonal matrix
satisfying $x\transp(QB_1 + B_1\transp Q) x\ge8\abs{x}^{2}$.
This is always possible, since $-B_{1}$ is a Hurwitz
lower diagonal matrix.
Then we have
\begin{align*}
b(x,u) \cdot  \grad \Lyap(x) &\;=\;\ell \cdot \grad \Lyap(x)
-\frac{m}{2} (x\transp Q x)^{\nicefrac{m}{2}-1} x\transp(QB_1 + B_1\transp Q) x \\
&\mspace{100mu} +m (x\transp Q x)^{\nicefrac{m}{2}-1}
Qx \bigl((B_1-\Gamma) (e\cdot x)^{+} u^c + B_2 (e\cdot x)^{-} u^s\bigr) \\[5pt]
&\;\le\;  m (\ell\transp Q x) (x\transp Q x)^{\nicefrac{m}{2}-1}
-  m (x\transp Q x)^{\nicefrac{m}{2}-1}
\bigl(4|x|^2 -C_1|x| |e\cdot x|\bigr)
\end{align*}
for some positive constant $C_1$.
Choosing $\delta = C_1^{-1}$ we obtain 
\begin{equation*}
b(x,u) \cdot  \grad \Lyap(x)  \;\le\;
C_2 -  m (x\transp Q x)^{\nicefrac{m}{2}-1} |x|^2
\qquad \forall\,x\in \cK_{\delta}^c\,,
\end{equation*}
for some positive constant $C_2$.
Similarly on the set $\cK_{\delta} \cap \{|x| \ge 1\}$, we can obtain the
following inequality
\begin{equation*}
b(x,u) \cdot  \grad \Lyap(x) \;\le\; C_3 (1+ |e\cdot x|^m)
\qquad \forall\,x\in \cK_{\delta}\,, 
\end{equation*}
for some positive constant $C_3>0$. 
Combining the above and rescaling $\Lyap$, we obtain 
\begin{equation*}
\Lg^{u}\Lyap (x) \;\le\; 1 - C_4 |x|^m \Ind_{\cK_{\delta}^c}(x)
+  C_{5} |e\cdot x|^m \Ind_{\cK_{\delta}}(x)\,, \quad x \in \RR^I\,,
\end{equation*}
for  some positive constants $C_4$ and $C_{5}$.  
Thus Hypothesis~\ref{HypA} is satisfied.
\end{proof}

\begin{remark}
It follows by Theorem~\ref{T-PA} that Lemma~\ref{L3.3} holds
for the ergodic control problem with constraints
in \eqref{diff-opt-c1}--\eqref{diff-opt-c2} under a running cost $r_{0}$ as in
\eqref{E-cost} with $\zeta\equiv 0$.
\end{remark}

\begin{theorem}\label{T-stab}
Suppose that the vector $\gamma$ is not identically zero.
There exists a constant Markov control $\Bar{u}=(\Bar{u}^c,\Bar{u}^s)\in\Act$
which is stable and has the following property:
For any $m\ge1$ there exists
a Lyapunov function $\Lyap$ of the form
$\Lyap(x) = (x\transp Q x)^{\nicefrac{m}{2}}$
for a diagonal positive matrix $Q$,
and positive constants $\kappa_{0}$ and $\kappa_{1}$ such that
\begin{equation}\label{E-stab}
\Lg^{\Bar{u}} \Lyap(x)\;\le\;
\kappa_{0}- \kappa_{1}\,\Lyap(x)\qquad \forall\,x\in\RI\,.
\end{equation}
As a result, the controlled process under $\Bar{u}$ is geometrically ergodic,
and its invariant probability distribution
has all moments finite.
\end{theorem}

\begin{proof}
Let $\him\in\cI$ be such that
$\gamma_{\him}>0$.
At each step of the algorithm the graph $\cG^{k}$ has at least two leaves
in $\cI$, unless it has maximum degree zero.
We eliminate the leaves in $\cI$ sequentially until we end up with a graph
consisting only of the edge $(\him,\hjm)$.
Then we set $\Bar{u}^c_{\him}=\Bar{u}^s_{\hjm}=1$.
This defines $\Bar{u}^c$ and $\Bar{u}^s$.
It is clear that $\Bar{u}= (\Bar{u}^c,\Bar{u}^s) \in \Act$. 
Note also that in the new ordering of the indices (replace with
the permutation $\pi$) we have $\him = I$
and and we can also let $\hjm = J$.

By construction (see also proof of Lemma~\ref{lem-drift}),  the drift takes
the form
\begin{equation*}
b_{i}(x, u_0) \;=\;
\begin{cases} \widehat b_{i}(x)\,, &\text{if}~i<I\,,\\[3pt]
\Tilde{b}_{I}(x_{1},\dotsc,x_{I-1}) - \mu_{IJ} x_{I}
-(\gamma_{I}-\mu_{IJ})\,\pex
+ \ell_{I} \,, &\text{if}~i=I\,,
\end{cases}
\end{equation*}
where $\widehat b$ is as in \eqref{E-hatb}. Note that the term
$(e\cdot x)^{-}$ does not appear in $b_i(x,u_0)$. 
The result follows by the lower-diagonal structure of the drift. 
\end{proof}

\begin{remark}
It is well known \cite[Lemma~2.5.5]{book} that
\eqref{E-stab} implies that
\begin{equation}\label{E-geom}
\Exp_{x}^{\Bar{u}}\bigl[\Lyap(X_{t})\bigr]\;\le\; \frac{\kappa_{0}}{\kappa_{1}}
+\Lyap(x)\,\E^{-\kappa_{1}t}\,, \qquad\forall x\in\RI\,,\quad \forall t\ge0\,.
\end{equation}
\end{remark}

\subsection{Special cases} \label{sec-special}
In the unconstrained control problems, we have assumed that the running cost
function $r(x, u)$ takes the form
in \eqref{E-cost}, where both the vectors $\xi$ and $\zeta$ are positive. 
However, if we were to select $\zeta\equiv 0$
(thus penalizing only the queue), 
then in order to apply the framework in
Section~\ref{sec-general}, we need to
verify Hypothesis~\ref{HypA} for a cone of the form
\begin{equation}\label{E-cKd}
\cK_{\delta,+}\;\df\;
\bigl\{ x\in\RI\,\colon e\cdot x > \delta \abs{x}\bigr\}\,,
\end{equation}
for some $\delta>0$. 
Hypothesis~\ref{HypA} relative to a cone $\cK_{\delta,+}$ implies
that, for some $\kappa>0$, we have
\begin{equation}\label{E-cKd2}
J_{v}\bigl[\nex\bigr]\;\le\;\kappa\, J_{v}\bigl[\pex\bigr]
\qquad \forall\,v\in\Usm\,.
\end{equation}
In other words, if under some Markov control the average queue length is finite,
then so is the average idle time.

Consider the ``W" model in Example~\ref{Ex2}.
When $e\cdot x<0$, the drift is 
\begin{equation*}
b(x,u)\;=\;- \begin{bmatrix} \mu_{11}  & 0 & 0 \\[5pt]
\mu_{21} (1+u^s_1) + \mu_{22} u^s_2 & \mu_{21} u^s_1 + \mu_{22} u^s_2 & 
\mu_{21} u^s_1 + \mu_{22} u^s_2 \\[5pt]
0 & 0 & \mu_{32} 
\end{bmatrix} x + \ell \,.
\end{equation*}
We leave it to the reader to verify that Hypothesis~\ref{HypA} holds
relative to a cone $\cK_{\delta,+}$ with a function $\Lyap$ of the
form $\Lyap(x) = (x\transp Q x)^{\nicefrac{m}{2}}$.
The same holds for the ``N" model, and the model in Example~\ref{Ex4}.

However for the ``M" model,
when $e\cdot x <0$, the drift takes the form
\begin{equation*}
b(x,u)\;=\;- \begin{bmatrix}  \mu_{12} (1-u^s_1) +  \mu_{11} u^s_1 
& (\mu_{11} - \mu_{12}) u^s_1 \\[5pt]
(\mu_{23} - \mu_{22}) u^s_3 & \mu_{22} (1- u^s_3) + \mu_{23}  u^s_3 
\end{bmatrix} x + \ell\,. 
\end{equation*}
Then it does not seem possible
to satisfy Hypothesis~\ref{HypA} relative to the cone
$\cK_{\delta,+}$,
unless restrictions on the parameters are imposed, for example,
if the service rates for each class do not differ much among the servers.
We leave it to the reader to verify that, provided
\begin{equation*}
\abs{\mu_{11} - \mu_{12}}\vee\abs{\mu_{23} - \mu_{22}}\;\le\;\tfrac{1}{2}\,
(\mu_{12}\wedge\mu_{22})\,,
\end{equation*}
Hypothesis~\ref{HypA} holds
relative to the cone $\cK_{\delta,+}$,
with $Q$ equal to the identity matrix. 
 An important implication from this example is that the ergodic
control problem may not be well posed if only the queueing cost is minimized
without penalizing the idleness either by including it in the running cost,
or by imposing constraints in the form of \eqref{diff-opt-c2}.

We present two results concerning special networks.

\begin{corollary}\label{C4.1}
Consider the ergodic control problem in \eqref{diff-opt}
with $X$ in \eqref{hatX} and $r(x,u)$ in \eqref{E-cost} with $\zeta\equiv 0$. 
For any $m\ge1$,  there exist positive
constants $\delta$, $\Tilde\delta$, and $\Tilde{\kappa}$, and a positive definite
$Q\in\RR^{I\times I}$ such that, if the service rates satisfy
\begin{equation*}
\max_{i\in\cI,\,j,k\in\cJ(i)}\;
\abs{\mu_{ij}-\mu_{ik}}\;\le\; \Tilde\delta\,
\max_{i\in\cI,\,j\in\cJ}\;\{\mu_{ij}\}\,,
\end{equation*}
then with $\Lyap(x) = (x\transp Q x)^{\nicefrac{m}{2}}$
\and $\cK_{\delta,+}$ in \eqref{E-cKd}  we have
\begin{equation*}
\Lg^{u}\Lyap(x) \;\le\; \Hat{\kappa} - \abs{x}^{m}\qquad
\forall x\in \cK_{\delta,+}^{c}\,,\quad \forall u\in\Act\,.
\end{equation*}
\end{corollary}

\begin{proof}
By \eqref{diff-map}, \eqref{diff-map2}
and \eqref{diff-drift}, if $\mu_{ij}=\mu_{ik}= \Bar{\mu}$ for all
$i\in\cI$ and $j,k\in\cJ$, then $b_{i}(x,u) = - \Bar{\mu}x_{i}$ when
$e\cdot x\le 0$, for all $i\in\cI$.
The result then follows by continuity.
\end{proof}

\begin{corollary}\label{C4.2}
Suppose there exists at most one $i\in\cI$ such that $\abs{\cJ(i)}>1$.
Then the conclusions of Corollary~\ref{C4.1} hold.
\end{corollary}

\begin{proof}
The proof follows by a straightforward application
of the leaf elimination algorithm.
\end{proof}

%

\begin{remark} \label{R4.6}
Consider the single-class multi-pool network (inverted ``V" model).
This model has been studied in \cite{Armony-05, AW-10}.
The service rates are pool-dependent, $\mu_j$ for $j \in \cJ$.
The limiting diffusion $X$ is one-dimensional.
It is easy to see from \eqref{diff-drift} that 
\begin{align*}
b(x,u) &\;=\;  x^{-} \sum_{j\in \cJ} \mu_j u^s_j - \gamma x^{+}  + \ell
\nonumber\\[5pt]
&\;=\; -\gamma x + x^{-} \Bigl( \sum_{j\in \cJ} \mu_j u^s_j + \gamma\Bigr) + \ell\,. 
\end{align*}
It can be easily verified that the controlled diffusion $X$ for this model
not only satisfies Hypothesis~\ref{HypA} relative to  $\cK_{\delta,+}$,
but it is positive recurrent under any Markov control,
and the set of invariant probability distributions
corresponding to stationary Markov controls is tight.
\end{remark}

\begin{remark}
Consider the multiclass multi-pool networks with class-dependent service rates,
that is, $\mu_{ij} = \mu_i$ for all $j \in \cJ(i)$ and $i \in \cI$.
In the leaf elimination algorithm, the sum of of the elements of
row $i$ of the matrix
$\Psi(\alpha, \beta)$ is equal to $\alpha_i$, for each $i\in \cI$.
Thus, by \eqref{diff-drift}, we have
\begin{equation*}
b_i(x,u)\;=\;b_i(x, (u^c, u^s))=  - \mu_i (x_i - (e\cdot x)^{+} u^c_i)
- \gamma_i (e\cdot x)^{+} u^c_i + \ell_i \qquad\forall\,i\in\cI\,. 
\end{equation*}
This drift is independent of $u^s$, and has the same form as the piecewise
linear drift studied in the multiclass single-pool model in \cite{ABP14}.
Thus,
the controlled diffusion $X$ for this model
satisfies Hypothesis~\ref{HypA} relative to  $\cK_{\delta,+}$.
Also Hypothesis~\ref{HypB}
holds for general running cost functions that are continuous,
locally Lipschitz and have at most polynomial growth, as shown in \cite{ABP14}. 
\end{remark} 

\medskip
\section{Characterization of Optimality} \label{sec-HJB}

In this section, we characterize the optimal controls via
the HJB equations associated with the ergodic control problems
for the limiting diffusions. 

\subsection{The discounted control problem}

The discounted control problem for the multiclass multi-pool network
has been studied in \cite{Atar-05a}.
The results strongly depend on estimates on moments of the controlled
process that are subexponential in the time variable.
We note here that the discounted infinite
horizon control problem is always solvable for the multiclass multi-pool
queueing network at the diffusion scale, without requiring any additional
hypotheses (compare with the assumptions in Theorem~1 of \cite{Atar-05a}).
Let $g\colon \RR^{I}\times\Act\to \RR_{+}$ be a continuous function,
which is locally Lipschitz in $x$ uniformly in $u$, and has at
most polynomial growth.
For $\theta>0$, define
\begin{equation}\label{E-disc}
\mathscr{J}_{\theta}(x;U)\;\df\;
\Exp^{U}_{x}\biggl[\int_{0}^{\infty} \E^{-\theta s}
g(X_{s},U_{s})\,\D{s}\biggr]\,.
\end{equation}
It is immediate by \eqref{E-geom} that $\mathscr{J}_{\theta}(x;\Bar{u})<\infty$
and that it inherits a polynomial growth from $g$.
Therefore $\inf_{U\in\Uadm}\;\mathscr{J}_{\theta}(x;U)<\infty$.
It is fairly standard then to show (see Section~3.5.2 in \cite{book})
that $V_{\theta}(x)\df\inf_{U\in\Uadm}\,\mathscr{J}_{\theta}(x;U)$
is the minimal nonnegative solution in
$\Cc^{2}(\RI)$ of the discounted HJB equation
\begin{equation*}
\frac{1}{2}\trace\bigl(\Sigma\Sigma\transp \nabla^{2} V_{\theta}(x)\bigr)
+ H(x,\nabla V_{\theta}) \;=\;\theta\, V_{\theta}(x)\,, \quad x\in\RI\,,
\end{equation*}
where
\begin{equation} \label{E-H}
H(x,p)\;\df\;\min_{u\in\Act}\;\bigl[b(x,u)\cdot p + g(x,u)\bigr]\,.
\end{equation}
Moreover, a stationary Markov control $v$ is optimal for the criterion
in \eqref{E-disc} if and only if it satisfies
\begin{equation*}
b\bigl(x,v(x)\bigr)\cdot\nabla V_{\theta}(x)
+ g\bigl(x,v(x)\bigr)\;=\;
H\bigl(x,\nabla V_{\theta}(x)\bigr)\quad\text{a.e.~in~} \RR^{I}\,.
\end{equation*}

\subsection{The HJB for the unconstrained problem}
\label{sec-HJBunconstrained}
The ergodic control problem for the limiting diffusion
falls under the general framework in \cite{ABP14}. 
We state the results for the existence of an optimal stationary Markov control,
and the existence and characterization of the HJB equation. 

Recall the definition of $J_{x,U}[r]$ and $\varrho^{*}(x)$
in \eqref{diff-cost}--\eqref{diff-opt}, and recall from
Section~\ref{sec-general} that if $v\in\Ussm$, then $J_{x,v}[r]$
does not depend on $x$ and is denoted by $J_{v}[r]$.
Consequently, if the ergodic control problem is well posed,
then $\varrho^{*}(x)$ does not depend on $x$.
We have the following theorem.

\begin{theorem}
There exists a stationary Markov control $v \in \Ussm$ that is optimal, i.e.,
it satisfies $J_{v}[r] = \varrho^*$. 
\end{theorem}

\begin{proof}
Recall 
that Hypothesis~\ref{HypA} is satisfied with $h(x) \df \Tilde{C} |x|^m$ for some
constant $\Tilde{C}>0$, as in
the proof of Theorem~\ref{T-PA}.
It is rather routine to verify that
\eqref{E-Lkey} holds for an inf-compact function $\Tilde{h}\sim|x|^m$.
The result then follows from Theorem~3.2 in \cite{ABP14}.
\end{proof}

We next state the characterization of the optimal solution via the associated
HJB equations. 

\begin{theorem} \label{thm-HJB}
For the  ergodic control problem of the limiting diffusion in \eqref{diff-opt},
the following hold:

\begin{itemize}
\item[(i)]
There exists a unique solution $V \in  \Cc^2(\RR^I)\cap \order(|x|^m)$,
satisfying $V(0)=0$, to the associated HJB equation: 
\begin{equation}\label{E-HJB}
\min_{u\in\Act}\;\bigl[\Lg^{u}V(x)+r(x,u)\bigr]
\;=\;\varrho^{*}\,.
\end{equation}
The positive part of $V$ grows no faster than $|x|^m$, and its negative part
is in $\sorder\bigl(\abs{x}^{m}\bigr)$.

\smallskip
\item[(ii)]
A stationary Markov control $v$ is optimal
if and only if it satisfies
\begin{equation}\label{E-v}
H\bigl(x,\nabla V(x)\bigr)
\;=\;b\bigl(x,v(x)\bigr)\cdot\nabla V(x)
+ r\bigl(x,v(x)\bigr)\quad\text{a.e.~in~} \RR^{I}\,,
\end{equation}
where $H$ is defined in \eqref{E-H}.

\smallskip
\item[(iii)]
The function $V$ has the stochastic representation
\begin{align*}
V(x)&\;=\;
\lim_{\delta\searrow0}\;\inf_{v\,\in\,\bigcup_{\beta>0}\,\Usm^{\beta}}\;
\Exp^{v}_{x}\biggl[\int_{0}^{\tc_{\delta}}
\bigl(r\bigl(X_{s},v(X_{s})\bigr)-\varrho^{*}\bigr)\,\D{s}\biggr]\\[5pt]
&\;=\;
\lim_{\delta\searrow0}\;\Exp^{\Bar{v}}_{x}\biggl[\int_{0}^{\tc_{\delta}}
\bigl(r\bigl(X_{s},v_{*}(X_{s})\bigr)-\varrho^{*}\bigr)\,\D{s}\biggr]\nonumber
\end{align*}
for any $\Bar{v}\in\Usm$ that satisfies \eqref{E-v}, where $v_*$
is the optimal Markov control satisfying \eqref{E-v}. 
\end{itemize}
\end{theorem}

\begin{proof}
The existence of a solution $V$ to the HJB \eqref{E-HJB} follows from
Theorem~3.4 in \cite{ABP14}.
It is facilitated by defining a running cost function
$r_{\varepsilon}(x, u)\,\df\, r(x,u) + \varepsilon \Tilde{h}(x,u)$ for
$\varepsilon>0$, and studying the corresponding ergodic control problem.
Uniqueness of the solution $V$ follows from Theorem~3.5 in \cite{ABP14}.  

The claim that the positive part of $V$ grows no faster than
$|x|^m$   follows from Theorems 4.1 and 4.2 in \cite{ABP14}, and the claim that
its negative part is in $\sorder\bigl(\abs{x}^{m}\bigr)$ follows from
Lemma~3.10 in \cite{ABP14}.

Parts (ii)--(iii) follow from Theorem~3.4 in \cite{ABP14}.
\end{proof}

For uniqueness of solutions to HJB,  see \cite[Theorem~3.5]{ABP14}.

The HJB equation in \eqref{E-HJB} can be also obtained via the traditional
vanishing discount approach. 
For $\alpha>0$ we define
\begin{equation} \label{E-V-alpha}
V_{\alpha}(x)\;\df\;\inf_{U\in\Uadm}\;
\Exp^{U}_{x}\biggl[\int_{0}^\infty \E^{-\alpha t}
r(X_{t},U_{t})\,\D{t}\biggr]\,.
\end{equation}
The following result follows directly from Theorem~3.6 of \cite{ABP14}.

\begin{theorem} \label{T5.3}
Let $V_{*}$ and $\varrho^{*}$ be as in Theorem~\ref{thm-HJB},
and let $V_{\alpha}$ be as in  \eqref{E-V-alpha}. 
The function $V_{\alpha} - V_{\alpha}(0)$ converges, as $\alpha\searrow0$,
to $V_{*}$, uniformly on compact subsets of $\RR^{I}$.
Moreover, $\alpha V_{\alpha}(0) \to \varrho^{*}$, as $\alpha\searrow0$.
\end{theorem}

The result that follows
concerns the approximation technique via spatial truncations of
the control.
For more details, including the properties of the associated
approximating HJB equations we refer the reader to \cite[Section~4]{ABP14}.

\begin{theorem}\label{T5.4}
Let $\Bar{u}\in\Act$ satisfy \eqref{E-stab}.
There exists a sequence $\{v_{k}\in\Ussm : k\in\NN\}$ such
that each $v_{k}$ agrees with $\Bar{u}$ on $B_{k}^{c}$, and
\begin{equation*}
J_{v_{k}}[r]\;\xrightarrow[k\to\infty]{}\;\varrho^{*}\,.
\end{equation*}
\end{theorem}

\begin{proof}
This follows by Theorems~4.1 and 4.2 in \cite{ABP14}, using the fact
that $\Tilde{h}\sim\Lyap\sim\abs{x}^{m}$.
\end{proof}

Since $\Act$ is convex, and
$r$ as defined in \eqref{E-cost} is convex in $u$, we have the following.

\begin{theorem}\label{C5.1}
Let $\Bar{u}\in\Act$ satisfy \eqref{E-stab}. 
Then, for any given $\varepsilon>0$, there exists an $R>0$ and an
$\varepsilon$-optimal continuous precise control
$v_{\varepsilon} \in \Ussm$ which is equal to $\Bar{u}$ on $B_{R}^{c}$.
In other words, if $\uppi_{v_{\varepsilon}}$ is the ergodic
occupation measure corresponding to $v_{\varepsilon}$, then 
\begin{equation*}
\uppi_{v_{\varepsilon}} (r)
\;=\; \int_{\Rd\times\Act} r(x,u)\, \uppi_{v_{\varepsilon}}(\D{x},\D{u}) \;\le\;
\varrho^* + \varepsilon\,.
\end{equation*}
\end{theorem}

\begin{proof}
Let $\Tilde{f}: \Act \to [0,1]$ be some strictly convex continuous function,
and define $r^\varepsilon(x,u) \df r(x,u) + \frac{\varepsilon}{3} \Tilde{f}(u)$, for
$\varepsilon>0$.
Let $\varrho^{*}_{\varepsilon}$ be the optimal value of the ergodic
problem with running cost $r^\varepsilon$.
It is clear that $\varrho^{*}_{\varepsilon}\le \varrho^{*} +\frac{\varepsilon}{3}$.

Let $v_{0}\in\Ussm$ be the constant control which is equal to $\Bar{u}$, and
for each $R \in \NN$, let $b^R(x,u)$ be as
defined in \eqref{bR} and analogously define $r_\varepsilon^{R}(x,u)$ as
in \eqref{rR} relative to $r_\varepsilon$.
Let $\Lg^{u}_{R}$ denote the controlled extended generator of the
diffusion with drift $b^R$ in \eqref{E-sdeR}.
Consider the associated HJB equation
\begin{equation}\label{E-HJB-R}
\min_{u\in\Act}\;\bigl[\Lg_{R}^{u}V_R(x)+r_\varepsilon^{R}(x,u)\bigr]
\;=\;\varrho(\varepsilon,R)\,.
\end{equation}
Since $u\mapsto \bigl[b^R(x,u)\cdot V_R+r_\varepsilon^{R}(x,u)\bigr]$
is strictly convex in $u$ for $x\in B_{R}$, and Lipschitz in $x$,
it follows that there is a (unique) continuous selector
$v_{\varepsilon,R}$ from the minimizer in \eqref{E-HJB-R}.
By Theorem~\ref{T5.4} (see also Theorems~4.1 and 4.2 in \cite{ABP14})
we can select $R$ large enough so that
\begin{equation}\label{ECT5.4a}
\varrho(\varepsilon,R)\,\le\, \varrho^{*}_{\varepsilon}+\frac{\varepsilon}{3}\,.
\end{equation}

Next we modify $v_{\varepsilon,R}$ so as to make it continuous on $\Rd$.
Let $\{\chi_k: k \in \NN\}$ be a sequence of cutoff functions such that
$\chi_k \in [0,1]$, $\chi_k \equiv 0$ on $B^c_{R-\frac{1}{k}}$,
and $\chi_k \equiv 1$ on $B_{R- \frac{2}{k}}$.
For $R$ fixed and satisfying \eqref{ECT5.4a},
define the sequence of controls
$\Tilde{v}_{k,\varepsilon} (x) \df
\chi_k(x) v_{\varepsilon,R}(x) + (1- \chi_k(x)) v_0(x)$,
and let $\uppi_{k}$ denote the associated sequence of ergodic occupation
measures. 
It is evident that $\Tilde{v}_{k,\varepsilon} \to v_{\varepsilon,R}$
in the topology of Markov controls \cite[Section~2.4]{book}.
Moreover, the sequence of measures $\uppi_{k}$ is tight, and 
therefore converges as $k\to\infty$ to the ergodic
occupation measure $\uppi_{\varepsilon}$ corresponding to $v_{\varepsilon,R}$
\cite[Lemma~3.2.6]{book}.
Since $r_\varepsilon$ is uniformly integrable with respect to the
sequence $\{\uppi_{k}\}$, we have
\begin{equation*}
\int_{\Rd\times\Act}r_\varepsilon(x,u)\,\uppi_{k}(\D{x},\D{u})
\;\xrightarrow[k\to\infty]{}\;\varrho(\varepsilon,R)\,.
\end{equation*}
Combining this with the earlier estimates finishes the proof.
\end{proof}

\subsection{The HJB for the constrained problem}
\label{sec-HJBconstrained}

As seen from Theorems~\ref{T3.1} and \ref{T-unique},
the dynamic programming formulation of the problem with constraints
in \eqref{diff-opt-c1}--\eqref{diff-opt-c2} follows in exactly the
same manner as the unconstrained problem.
Also, Theorems~\ref{T3.3} and \ref{T3.4} apply.

We next state the analogous results of Theorems~\ref{T5.4}
and~\ref{C5.1} for the constrained problem.

\begin{theorem} \label{T5.5}
Let $\Bar{u}\in\Act$ satisfy \eqref{E-stab}.
Suppose that $\updelta\in (0,\infty)^{J}$ is feasible.
Then there exist  $k_0 \in \NN$ and a sequence
$\{v_k\in \Ussm\,\colon\, k \in \NN\}$
such that for each $k \ge k_0$, $v_k$ is equal to $\Bar{u}$ on $B^c_k$ and 
\begin{equation*}
J_{v_k}[r_0]
\;\xrightarrow[k\to\infty]{}\; \varrho_{c}^{*}\;=\;
\inf_{\uppi\,\in\,\sH(\updelta)}\;\uppi(r_{0}) \,.
\end{equation*}

\end{theorem}

\begin{proof}
This follows from Theorem~\ref{T3.4}. 
\end{proof}

Since $r_j(x,u)$ defined in \eqref{E-rj} is convex in $u$ for
$j=0,1,\dotsc,J$, we have the following.

\begin{theorem} \label{C5.2}
Let $\Bar{u}\in\Act$ satisfy \eqref{E-stab}.
Suppose that $\updelta\in (0,\infty)^{J}$ is feasible.
Then for any given $\varepsilon>0$, there exists $R_{0}>0$
and a family continuous precise controls
$v_{\varepsilon,R} \in \Ussm$, $R>R_{0}$
satisfying the following:
\begin{itemize}
\item[(i)]
Each $v_{\varepsilon,R}$ is equal to $\Bar{u}$ on $B_{R}^{c}$.

\smallskip
\item[(ii)]
The corresponding ergodic occupation measures
$\uppi_{v_{\varepsilon,R}}$ satisfy
\begin{subequations}
\begin{align}
\uppi_{v_{\varepsilon,R}}(r_0) &\;\le\; \varrho_{c}^{*}
+ \varepsilon\qquad\forall\,R>R_{0}\,,\label{EC5.2a}\\[5pt]
\sup_{R>R_{0}}\;\uppi_{v_{\varepsilon,R}}(r_j)&\;<\; \updelta_j\,,\qquad
j\in\cJ\,.\label{EC5.2b}
\end{align}
\end{subequations}
\end{itemize}
\end{theorem}

\begin{proof}
By Lemma~\ref{L3.4}, for all sufficiently small $\varepsilon>0$, there exist
$\updelta_{j}^{\varepsilon}<\updelta_{j}$,
$j\in\cJ$, such that $\updelta^{\varepsilon}$ is feasible and
\begin{equation*}
\inf_{\uppi\,\in\,\sH(\updelta^{\varepsilon})}\;\uppi(r_{0})
\;\le\; \inf_{\uppi\,\in\,\sH(\updelta)}\;\uppi(r_{0})
+ \tfrac{\varepsilon}{4}\,.
\end{equation*}
Let
\begin{equation*}
g^{\varepsilon}_{\updelta,\uplambda}(x,u)\df g_{\updelta,\uplambda}(x,u)
+ \tfrac{\varepsilon}{3}\,\Tilde{f}(u)\,,\qquad\uplambda\in \RR_+^J\,,
\end{equation*}
where $\varepsilon>0$, $g_{\updelta,\uplambda}$, is as in Definition~\ref{D3.4},
and $\Tilde{f}: \Act \to [0,1]$ is some strictly convex continuous function.
Let $v_{0}\in\Ussm$ be the constant control which is equal to $\Bar{u}$, and
for each $R \in \NN$, let $b^R(x,u)$ be as
defined in \eqref{bR}.
Recall the definition of $\eom(R)$ and $\sH(\updelta;R)$ in the
paragraph preceding Theorem~\ref{T3.3}.
By Theorem~\ref{T3.4}, there exists $\uplambda^{*}_{R} \in \RR_+^J$ such that
\begin{equation*}
\inf_{\uppi\,\in\,\eom(R)}\;
\uppi(g^{\varepsilon}_{\updelta^{\varepsilon},\uplambda^{*}_{R}})
\;=\;\inf_{\uppi\,\in\,\sH(\updelta^{\varepsilon};R)}\; \uppi
\bigl(r_{0}+\tfrac{\varepsilon}{4}\,\Tilde{f}\bigr)\,,
\end{equation*}
and \eqref{E-HJB-n} holds, and moreover, $R>0$ can be selected large
enough so that
\begin{align*}
\inf_{\uppi\,\in\,\sH(\updelta^{\varepsilon};R)}\; \uppi
\bigl(r_{0}+\tfrac{\varepsilon}{4}\,\Tilde{f}\bigr)
&\;\le\; \inf_{\uppi\,\in\,\sH(\updelta^{\varepsilon})}\; \uppi
\bigl(r_{0}+\tfrac{\varepsilon}{4}\,\Tilde{f}\bigr)+\tfrac{\varepsilon}{4}\\[5pt]
&\;\le\; \inf_{\uppi\,\in\,\sH(\updelta^{\varepsilon})}\; \uppi
\bigl(r_{0}\bigr)+\tfrac{\varepsilon}{2}\,.
\end{align*}
Combining these estimates we obtain
\begin{equation*}
\inf_{\uppi\,\in\,\eom(R)}\;
\uppi(g^{\varepsilon}_{\updelta^{\varepsilon},\uplambda^{*}_{R}})
\;\le\; \inf_{\uppi\,\in\,\sH(\updelta)}\;\uppi(r_{0})
+ \tfrac{3\varepsilon}{4}\,.
\end{equation*}
By strict convexity there exists a (unique) continuous selector
$v_{\varepsilon,R}$ from the minimizer in \eqref{E-HJB-n}.
Using a cutoff function $\chi$ as in the proof of
Theorem~\ref{C5.1}, and redefining
completes the argument. 
\end{proof}

\subsection{Fair allocation of idleness}
\label{sec-fair}
There is one special case of the ergodic problem
under constraints which is worth investigating further.
Let
\begin{equation*}
\cS^J\;\df\;\{z\in\RR^{J}_{+}\;\colon\, e\cdot z=1\}\,.
\end{equation*}
Consider the following assumption.

\begin{assumption}\label{A-fair}
Hypothesis~\ref{HypA} holds relative to a cone
$\cK_{\delta,+}$ in \eqref{E-cK}, and for every
$\Hat u^s\in\cS^J$ there exists a stationary Markov control
$v(x)=(v^{c}(x),\Hat u^s)$ such that $J_{v}[r_{0}]<\infty$.
\end{assumption}

Examples of networks that Assumption~\ref{A-fair} holds  were discussed in
Section~\ref{sec-example}.
In particular, it holds for the ``W'' network, the network in Example~\ref{Ex4},
and in general under the hypotheses of Corollaries~\ref{C4.1} and \ref{C4.2}.

Let $r_{0}(x,u)$ be as defined in \eqref{E-cost} with $\zeta\equiv 0$, and
\begin{equation*}
r_j(x,u)\;\df\;(e\cdot x)^{-}u^s_j\,,\quad j\in\cJ\,,
\end{equation*}
Let $\uptheta$ be an interior point of $\cS^{J}$,
i.e., $\uptheta_{j}>0$ for all $j\in\cJ$,
and consider the problem with constraints given by 
\begin{align} 
\varrho_{c}^{*} &\;=\; \inf_{v\,\in\,\Ussm} \;J_{v}[r_{0}]\label{E-fair1}
\\[5pt]
\text{subject to}\quad J_{v}[r_j] &\;=\; \uptheta_j\,\sum_{k=1}^{J}J_{v}[r_k]\,,
\quad j=1,\dotsc,J-1\,.
\label{E-fair2} 
\end{align}
The constraints in \eqref{E-fair2} impose fairness on idleness.
In terms of ergodic occupation measures, the problem takes the form
\begin{align} 
\varrho_{c}^{*} &\;=\; \inf_{\uppi\,\in\,\eom} \;\uppi(r_{0})\label{E-fair3}\\[5pt]
\text{subject to}\quad
\uppi(r_{j}) &\;=\; \uptheta_j\,\sum_{k=1}^{J}\uppi(r_{k})\,,
\quad j=1,\dotsc,J-1\,.
\label{E-fair4} 
\end{align}
Following the proof of Lemma~\ref{L3.3}, using \eqref{E-cKd2}
and Assumption~\ref{A-fair}, we deduce
that the infimum in \eqref{E-fair3}--\eqref{E-fair4} is finite,
and is attained at some $\uppi^{*}\in\eom$.
Define
\begin{equation*}
L(\uppi,\uplambda) \;\df\; \uppi(r_{0}) + \sum_{j=1}^{J-1}
\uplambda_{j} \Bigl(  \uppi(r_{j})-\uptheta_{j}
\,\sum_{k=1}^{J}\uppi(r_{k}) \Bigr) \,.
\end{equation*}
We have the following theorem.

\begin{theorem}\label{T-fair}
Let Assumption~\ref{A-fair} hold.
Then for any $\uptheta$ in the interior of $\cS^{J}$ there exists
a $v^{*}\in\Ussm$ which is optimal for the ergodic cost problem
with constraints in \eqref{E-fair1}--\eqref{E-fair2}.
Moreover, there exists $\uplambda^{*}\in\RR^{J-1}_{+}$ such that
\begin{equation*}
\varrho_{c}^{*} \;=\;
\inf_{\uppi\,\in\,\eom}\;L(\uppi,\uplambda^{*})\,,
\end{equation*}
and $v^{*}$ can be selected to be a precise control.
\end{theorem}

\begin{proof}
The proof is analogous to the one in Lemma~\ref{L3.5}.
It suffices to show that the constraint is linear and feasible
(see also \cite[Problem~7, p.~236]{Luenberger}).
Let $\Tilde\eom\df\{\uppi\in\eom : \uppi(r_{0})<\infty\}$.
By the convexity of the set of ergodic occupation measures, it follows
that $\Tilde\eom$ is a convex set.
Consider the map $F:\Tilde\eom\to\RR^{J-1}$ given by
\begin{equation*}
F_{j}(\uppi)\;\df\;
\uppi(r_{j}) - \uptheta_j\,\sum_{k=1}^{J}\uppi(r_{k})\,,\quad j=1,\dotsc,J-1\,.
\end{equation*}
The constraints in \eqref{E-fair4} can be written as $F(\uppi)=0$
and therefore are linear.

We claim that $0$ is an interior point of $F(\Tilde\eom)$.
Indeed, since $\uptheta$ be an interior point of $\cS^{J}$,
for each $\hjm\in\{1,\dotsc,J-1\}$ we may select
$\Hat{u}^{s}\in\cS^{J}$ such that
$\Hat{u}^{s}_{j} =\uptheta_{j}$ for $j\in\{1,\dotsc,J-1\}\setminus\{\hjm\}$,
and $\Hat{u}^{s}_{\hjm}>\uptheta_{\hjm}$.
By Assumption~\ref{A-fair}, there exists $v\in\Ussm$, of the
form $v=(v^c,\Hat{u}^s)$ such that $\uppi_{v}\in\Tilde\eom$.
It is clear that
$F_{j}(\uppi_{v})=0$ for $j\ne\hjm$, and $F_{\hjm}(\uppi_{v})>0$.
Repeating the same argument with $\Hat{u}^{s}_{\hjm}<\uptheta_{\hjm}$
we obtain $\uppi_{v}\in\Tilde\eom$ such that
$F_{j}(\uppi_{v})=0$ for $j\ne\hjm$, and $F_{\hjm}(\uppi_{v})<0$.
Thus we can construct a collection
$\Tilde\eom_{0}=\{\Tilde\uppi_{1},\dotsc,\Tilde\uppi_{2J-2}\}$ of elements of
$\Tilde\eom$ such that $0$ is an interior point of the convex hull
of $F(\Tilde\eom_{0})$.
This proves the claim, and the theorem.
\end{proof}

\begin{remark}
Theorem~\ref{T-fair} remains of course valid if fewer than
$J-1$ constraints, or no constraints at all are imposed,
in which case the assumptions can be weakened.
For example, in the case of no constraints, we only require that
Hypothesis~\ref{HypA} holds relative to a cone
$\cK_{\delta,+}$ in \eqref{E-cK}, and the results
reduce to those of Theorem~\ref{thm-HJB}.

Also, the dynamic programming counterpart of Theorem~\ref{T-fair}
is completely analogous to Theorem~\ref{T3.1},
and the conclusions of Theorems~\ref{T5.5} and~\ref{C5.2}
hold.
\end{remark}

\section{Conclusion} \label{sec-conclusion}

We have developed a new framework to study the (unconstrained and constrained)
ergodic diffusion control problems for Markovian multiclass multi-pool networks in the
Halfin--Whitt regime. The explicit representation for the drift of the limiting
controlled diffusions, resulting from the recursive
leaf elimination algorithm of tree bipartite
networks, plays a crucial role in establishing the
needed positive recurrence properties of the limiting diffusions.
These results are relevant to the recent study of the stability/recurrence properties
of the multiclass multi-pool networks in the Halfin--Whitt regime under
certain classes of control policies
\cite{stolyar-yudovina-12b, stolyar-yudovina-12a, stolyar-14,stolyar-15}. 
The stability/recurrence properties for general multiclass multi-pool networks
under other scheduling policies remain open.
It is important to note that our approach to ergodic control of
these networks does not, a priori, rely on any uniform stability properties
of the networks.

We did not include in this paper any asymptotic optimality results
of the control policies constructed from the HJB equation in the Halfin-Whitt regime.
We can establish the lower bound following the method in \cite{ABP14} for the
``V" model, albeit with some important differences in technical details.
The upper bound is more challenging.
What is missing here, is a result analogous to Lemma~5.1 in \cite{ABP14}.
Hence we leave asymptotic optimality as the subject of a future paper.

The results in this paper may be useful to study other diffusion control
problems of multiclass multi-pool networks in the Halfin--Whitt regime. 
The methodology developed for the ergodic control of diffusions for such networks
may be applied to study other classes of stochastic networks;
for example, it remains to study ergodic control problems for multiclass
multi-pool networks that do not have a tree structure and/or have feedback.
This class of ergodic control problems of diffusions may
also be of independent interest to the ergodic control literature.
It would be interesting to study numerical algorithms, such as the
policy or value iteration schemes, for this class of models.

\section*{Acknowledgements}
The work of Ari Arapostathis was supported in part by the Office of Naval
Research through grant N00014-14-1-0196, and in part by a grant from
the POSTECH Academy-Industry Foundation.
The work of Guodong Pang is supported in part by the Marcus Endowment Grant at the 
Harold and Inge Marcus Department of Industrial and Manufacturing Engineering 
at Penn State.

\end{document}